\documentclass[10pt]{amsart}

\usepackage[margin=1.13in]{geometry}
\usepackage{amsthm,amsmath,amsfonts,amssymb}
\usepackage{graphicx}
\usepackage{dsfont}

\usepackage[usenames,dvipsnames,svgnames,table]{xcolor}

\usepackage[colorlinks=true,
pdfstartview=FitV,linkcolor=ForestGreen,citecolor=ForestGreen,
urlcolor=blue]{hyperref}

\newtheorem{thm}{Theorem}[section]
\newtheorem{prop}[thm]{Proposition}

\newtheorem{lemma}[thm]{Lemma}
\newtheorem{defn}[thm]{Definition}
\theoremstyle{remark}
\newtheorem{remark}[thm]{Remark}
\newtheorem{remarks}[thm]{Remarks}
\numberwithin{equation}{section}

\newcommand{\R}{\mathbb R}
\newcommand{\Sp}{\mathbb S^{d-1}}

\newcommand{\eps}{\varepsilon}
\newcommand{\grad} {\nabla}

\newcommand{\dd} {\; \mathrm{d}}

\DeclareMathOperator{\PV}{\mbox{p.v.}}

\newcommand{\angulars}{{2s}}

\newcommand{\one}{\mathds{1}}
\newcommand{\A}{N}
\newcommand{\Ao}{N_0}
\newcommand{\Ai}{N_\infty}

\title[Decay estimates in the Boltzmann without cutoff]{Decay
  estimates for large velocities in the Boltzmann equation without
  cutoff}

\author{Cyril Imbert}
\address[C.~Imbert]{CNRS \& Department of Mathematics and
  Applications, \'Ecole Normale Sup\'erieure (Paris) \\
45 rue d'Ulm, 75005 Paris, France}
\email{Cyril.Imbert@ens.fr}
\author{Cl\'ement Mouhot}
\address[C.~Mouhot]{University of Cambridge,
      DPMMS, Centre for Mathematical Sciences,
      Wilberforce road, Cambridge CB3 0WA, UK}
\email{C.Mouhot@dpmms.cam.ac.uk}
\author{Luis Silvestre} 
\thanks{LS is supported in part by NSF grants DMS-1254332 and
  DMS-1362525.
  CM is partially supported by ERC grant MAFRAN}
\address[L.~Silvestre]{Mathematics
    Department, University of Chicago, Chicago, Illinois 60637, USA}
\email{luis@math.uchicago.edu}

\date{\today}

\begin{document}

\begin{abstract}
  We consider solutions $f=f(t,x,v)$ to the full (spatially
  inhomogeneous) Boltzmann equation with periodic spatial conditions
  $x \in \mathbb T^d$, for hard and moderately soft potentials
  \emph{without the angular cutoff assumption}, and under the \emph{a
    priori} assumption that the main hydrodynamic fields, namely the
  local mass $\int_v f(t,x,v)$ and local energy $\int_v f(t,x,v)|v|^2$
  and local entropy $\int_v f(t,x,v) \ln f(t,x,v)$, are controlled
  along time.  We establish quantitative estimates of
  \emph{propagation} in time of ``pointwise polynomial moments'', i.e.
  $\sup_{x,v} f(t,x,v) (1+|v|)^q$, $q \ge 0$. In the case of hard
  potentials, we also prove \emph{appearance} of these moments for all
  $q \ge 0$. In the case of moderately soft potentials we prove the
  \emph{appearance} of low-order pointwise moments.
\end{abstract}
\maketitle

\setcounter{tocdepth}{1}
\tableofcontents

\section{Introduction}

\subsection{The Boltzmann equation}

The Boltzmann equation models the evolution of rarefied gases,
described through the probability density of the particles in the
phase space. It sits at a mesoscopic scale between the hydrodynamic
equations (e.g. the compressible Euler or Navier Stokes equations)
describing the evolution of observable quantities on a large scale,
and the complicated dynamical system describing the movement of the
very large number of molecules in the gas. Fluctuations around steady
state, on a large scale, follow incompressible Navier-Stokes equations
under the appropriate limit.

This probability density of particles is a non-negative function
$f=f(t,x,v)$ defined on a given time interval $I \in \R$ and
$(x,v) \in \R^d \times \R^d$ and it solves the integro-differential
\emph{Boltzmann equation}
\begin{equation}\label{eq:boltzmann}
  \partial_t f + v \cdot \grad_x f = Q (f,f). 
\end{equation}
The bilinear \emph{Boltzmann collision operator} $Q(f_1,f_2)$ is defined
as
\begin{equation*}
  Q (f_1,f_2) := \int_{\R^d} \int_{\Sp} \Big[ f_1(v'_*)f_2(v') - f_1(v_*)f_2(v)
  \Big] B(|v-v_*|,\cos \theta) \dd v_* \dd \sigma
\end{equation*}
where $B$ is the \emph{collision kernel} and the \emph{pre-collisional
  velocities} $v_*'$ and $v'$ are given by (see
Figure~\ref{fig:collision})
\[ v' := \frac{v+v_*}2 + \frac{|v-v_*|}2 \sigma \quad \text{ and }
  \quad v'_* := \frac{v+v_*}2 - \frac{|v-v_*|}2 \sigma. \] The
so-called \emph{deviation angle} $\theta$ is the angle between the
pre- and post-collisional relative velocities (observe that
$|v-v_*|=|v'-v'_*|$):
\begin{align*}
& \cos \theta := \frac{v-v_*}{|v-v_*|} \cdot
  \frac{v'-v'_*}{|v'-v'_*|} \quad  \mbox{and} \quad \sigma :=
  \frac{v'-v'_*}{|v'-v'_*|} \\
 &\sin (\theta/2)  := \frac{v'-v}{|v'-v|} \cdot \sigma \quad  \mbox{and} \quad  
\cos (\theta/2)  := \frac{v'-v_*}{|v'-v_*|} \cdot \sigma .
\end{align*}

\begin{figure}
\includegraphics[height=6.5cm]{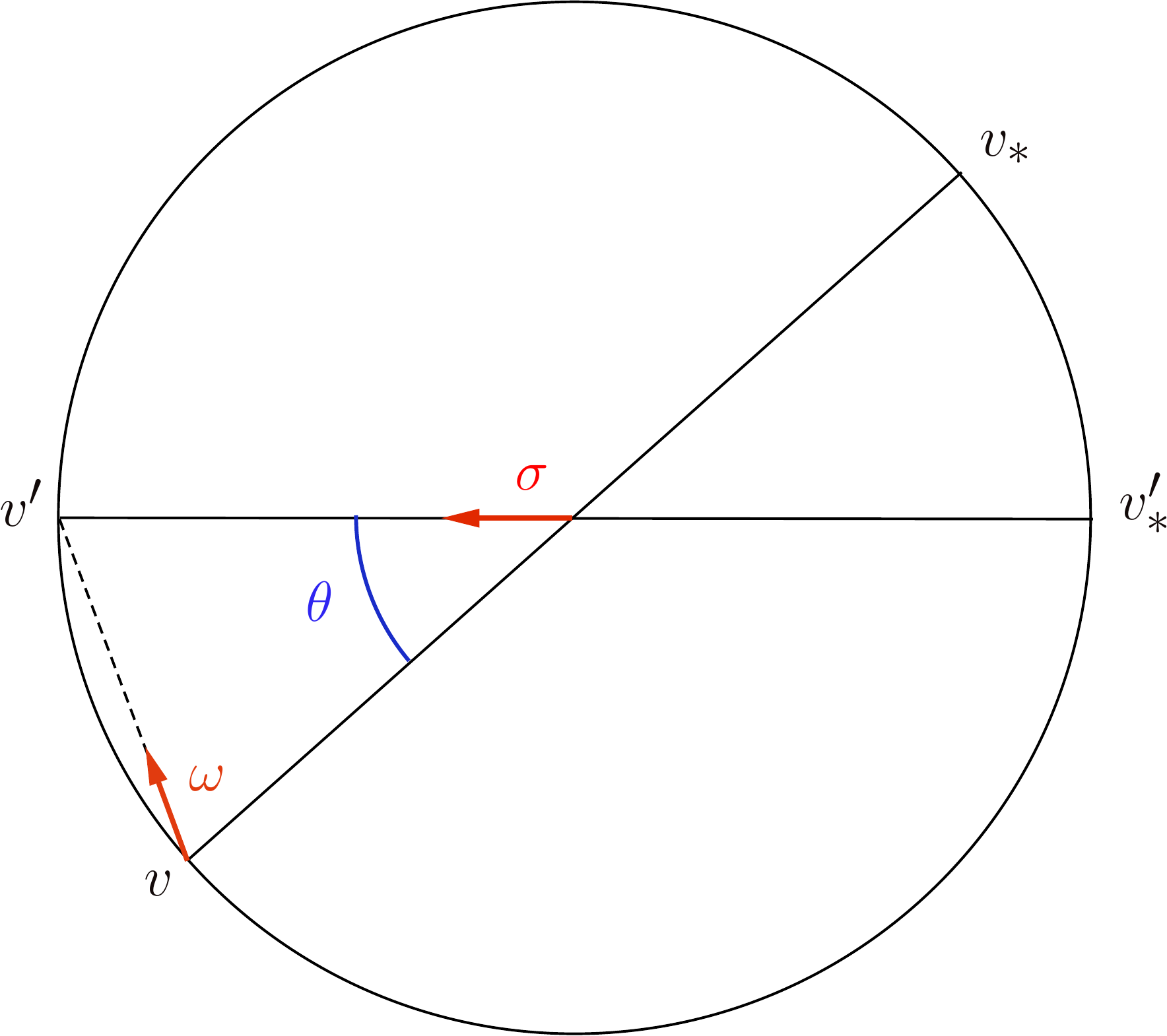}
\caption{The geometry of the binary collision.}
\label{fig:collision}
\end{figure}

The precise form of the collision kernel $B$ depends on the molecular
interaction \cite{MR1313028}. For all \emph{long-range interactions},
that is all interactions apart from the \emph{hard spheres model}, it
is singular at $\theta \sim 0$, i.e. small deviation angles that
correspond to \emph{grazing collisions}. Keeping this singularity
dictated by physics in the mathematical analysis has come to be known
quite oddly as a \emph{non-cutoff} assumption. In dimension $d=3$,
when this long-range interaction derives from a power-law repulsive
force $F(r) = C r^{-\alpha}$ with $\alpha \in (2,+\infty)$, then $B$
is given by (see~\cite{MR1313028} and \cite[Chapter~1]{villani-book})
\begin{equation*}
B(r,\cos \theta) = r^\gamma b(\cos \theta) \quad \text{ with} \quad b
(\cos \theta) \sim_{\theta \sim 0} \mbox{cst} \, \theta^{-(d-1)-2s}
\end{equation*}
with $\gamma = \frac{\alpha-5}{\alpha-1} \in (-d,1)$, $C>0$ and
$s = \frac{1}{\alpha-1} \in (0,1)$. The singularity of $b$ at
\emph{grazing collisions} $\theta \sim 0$ is the legacy of long-range
interactions.

The assumption above is vaguely formulated as far as the singularity
is concerned, the more precise formulation that we will use in this
paper is
\begin{equation}\label{assum:B}
  \begin{cases}
    B(r,\cos \theta)  = r^\gamma b(\cos \theta) \\[2mm] 
    b(\cos \theta)  = (\sin \theta/2)^{-(d-2)+\gamma} (\tan
    \theta/2)^{-(\gamma+2s+1)} \tilde b(\cos \theta)
  \end{cases}
\end{equation}
with some smooth $\tilde b$ satisfying
$0 < \tilde b_0 \le \tilde b \le \tilde b_1$ for constants
$\tilde b_0$, $\tilde b_1>0$. 
The precise mixture of sinus and tangent functions to model the
singularity is made for technical conveniency and is no loss of
generality: it is easy to check using the symmetry of the collision
process $v' \leftrightarrow v'_*$ that physical collision kernels
satisfy it. It corresponds to the technical condition (3.1) in \cite{luis}. 

We consider collision kernels satisfying~\eqref{assum:B} in general
dimension $d \ge 2$ and with general exponents $\gamma \in (-d,2]$ and
$s \in (0,1)$ that are not necessarily derived from the inverse
power-law formula above. The hard spheres interactions play the role
of the limit case $\alpha \to \infty$ ($\gamma=1$ and integrable $b$).

It is standard terminology in dimension $d=3$ to denote respectively:\\
- the case $\alpha=5$ ($\gamma=0$ and $2s=1/2$) as \emph{Maxwell
  molecules} \cite{Maxwell1867},\\
- the case $\alpha \in (5,+\infty)$ ($\gamma \in (0,1)$ and
$2s \in (0,1/2)$) as \emph{hard potentials},\\
- the case $\alpha \in [3,5)$ ($\gamma \in [-1,0)$ and
$2s \in (1/2,1]$) as \emph{moderately soft potentials},\\
- the case $\alpha \in (2,3)$ ($\gamma \in (-3,-1)$ and
$2s \in (1,2)$) as \emph{very soft potentials}.

The limit $s \to 1$ is called the \emph{grazing collision limit}, and
in this limit the Boltzmann collision operator converges to
the \emph{Landau-Coulomb collision operator}. It turns out that the
threshold between moderately and very soft potentials corresponds to
$\gamma + 2s =0$. We therefore denote by \emph{moderately soft
  potentials}, in any dimension $d \ge 2$, the case
$\gamma + 2s \in [0,2]$.

\subsection{The question at hand}

The global well-posedness for solutions to the inhomogeneous Boltzmann
equation is an outstanding open problem. Since it is a more detailed
model than the Euler and Navier-Stokes equations, and it includes
these equations as limits in certain scalings, one can expect that it
will share some of the (currently intractable) difficulties of these
hydrodynamic models. Even in the spatially homogeneous case, the
Cauchy problem is shown to be well-posed without perturbative
assumptions only in the case of moderately soft potentials
\cite{MR2525118}. Given that global well-posedness seems out of reach
at present time, our more realistic goal is to show that for suitable
initial data $f(0,x,v) = f_0(x,v)$, the equation \eqref{eq:boltzmann}
has a unique smooth solution for as long as its associated
hydrodynamic quantities stay under control. Morally, this neglects the
hydrodynamic difficulties of the model and concentrates on the
difficulties that are intrinsic to the kinetic representation of the
fluid.

Let us state the longer-term conjecture. Consider the following
hydrodynamic quantities
\begin{align*}
  \text{(mass density)} \qquad M (t,x) &:= \int_{\R^d} f(t,x,v) \dd v, \\
  \text{(energy density)} \qquad E (t,x) & := \int_{\R^d} f(t,x,v) |v|^2 \dd v, \\
  \text{(entropy density)} \qquad H (t,x) & := \int_{\R^d}  f \ln f (t,x,v) \dd v. 
\end{align*} 

\medskip

\noindent
\noindent {\bf Conjecture (conditional regularisation).} Consider any
solution
\[
  f \in L^\infty([0,T],L^1(\R^d \times \R^d))
\]
to \eqref{eq:boltzmann} on a time interval $[0,T]$ for some
$T \in (0,+\infty]$, such that the hydrodynamic fields of $f$ remain
controlled on this time interval: more precisely assume that for all
points $(t,x)$, the mass density is bounded below and above
$0 < m_0 \leq M(t,x) \leq M_0$, the energy density is bounded above
$E(t,x) \leq E_0$ and the entropy density is bounded above
$H(t,x) \leq H_0$ (for constants $m_0$, $M_0$, $E_0$, $H_0>0$). Then
this solution is bounded and smooth on
$(0,T]$.% , and is unique within
% the class of solutions satisfying these hydrodynamic bounds and
% minimal integrability assumptions.
\medskip

\begin{remarks}
  \begin{enumerate}
  \item Observe that the contraposition of this statement means that
    any finite-time blow-up in solutions to the Boltzmann equation
    with long-range interactions must include a blow-up in the
    hydrodynamic quantity (local mass, energy or entropy diverging at
    some position), or the creation of vacuum (local mass vanishing at
    some position). In other words, one of the hydrodynamic bounds
    above has to degenerate as $t \uparrow T^-$.

  \item There are two natural ways in which this conjecture can be
    strengthened or weakened:
    \begin{enumerate}
    \item Strengthening the statement: the blow-up scenario through
      the creation of vacuum is likely to be ruled out by further
      work, which means that the lower bound assumption on the mass
      could be removed. Mixing in velocity through collisions combined
      with transport effects generate lower bounds in many settings,
      see \cite{MR2153518,MR2746671,MR3375544,MR3356579}, and the
      assumption was indeed removed for the related Landau equation
      with moderately soft potentials in \cite{henderson2017local}.
      We might also expect that the pointwise bounds could be replaced
      with an $L_t^p(L_x^q)$ bound for $E$, $M$ and $H$, similar to
      the Prodi-Serrin condition for Navier-Stokes equations.
    \item Weakening the statement: more regularity or decay could be
      assumed on the initial data, as long as it is propagated
      conditionally to the hydrodynamic bounds assumed on the
      solution. This would slightly weaken the conjecture but the
      contraposed conclusion would remain unchanged: any blow-up must
      occur at the level of the hydrodynamic quantities.
    \end{enumerate}
  \end{enumerate}
\end{remarks}  

\subsection{Known results of conditional regularisation in kinetic
  theory}

\subsubsection{The Boltzmann equation with long-range interactions} 
In \cite{amuxy-arma2010}, the authors prove that if the solution $f$ has
five derivatives in $L^2$, with respect to all variables $t$, $x$ and
$v$, weighted by $\langle v \rangle^q := (1+|v|^2)^{q/2}$ for
arbitrarily large powers $q$, and in addition the mass density is
bounded below, then the solution $f$ is $C^\infty$. Note also that
stability (uniqueness) holds under such
$H^5_{x,v}(\langle v \rangle^q)$ regularity. Note also the previous
partial result \cite{MR2038147} and the subsequent follow-up papers
\cite{MR2149928,MR2425608,MR2462585,MR2476677,MR2476686} in the
spatially homogeneous case, with less assumption on the initial
data. Our goal however is to reduce the regularity assumed on the
solution as close to the minimal hydrodynamic bounds as possible.

The natural strategy we follow goes through the following steps:
\begin{enumerate}
\item \textbf{A pointwise estimate in
  $L^\infty((0,T] \times \R^d \times \R^d)$}: observe that hydrodynamic
  quantities only control $v$-integrals on the solution.
\item \textbf{A decay estimate for large velocities}: the non-compact
  velocity space is a source of mathematical difficulties in the
  Boltzmann theory, and badly thermalised solutions (e.g. spikes of
  high-velocity particles) break regularity estimates. Such decay can
  be searched in $L^1$ (moment estimates) or $L^\infty$ spaces as in
  this paper.
\item \textbf{A regularisation estimate in H\"older spaces}: this is
  where the hypoelliptic nature of the equation enters the strategy,
  and such a regularity estimate is in the spirit of De
  Giorgi-Nash-Moser theory.
\item \textbf{Schauder estimates to obtain higher regularity by
    bootstrap}: this is a standard principle for quasilinear equations
  that regularity can be bootstrapped in $C^{k}$ H\"older spaces, but
  the non-local integral nature of the collision operator creates new
  interesting difficulties.
\end{enumerate}
  
The first step was completed in \cite{luis}. The main result in the
present paper is the completion of the second step, i.e. decay
estimate for large velocities. The third step, i.e. the regularisation
in $C^\alpha$ was completed in \cite{is}. The bootstrap mechanism to
obtain higher regularity is the piece of the puzzle that currently
remains unsolved. In future work, we intend to address the forth step using the Schauder estimates from \cite{imbert2018schauder}.

\subsubsection{The Landau equation} This program of ``conditional
regularisation'' following the four steps above has already been
carried out for the inhomogeneous Landau equation with moderately soft
potentials, which corresponds to the limit of the Boltzmann equation
as $s \to 1$, when furthermore $\gamma \in [-2,0]$. The $L^\infty$
estimate, as well as Gaussian upper bounds, were obtained in
\cite{css} (first and second steps). The regularisation estimate in
H\"older spaces was obtained in \cite{gimv} (third step). The fourth
step was completed in \cite{henderson2017c} in the form of Schauder
estimates for kinetic parabolic equations. The regularity of solutions
of the Landau equation is iteratively improved using Schauder
estimates up to $C^\infty$ regularity. In the physical case of the
Landau-Coulomb equation (playing the role of the limit case
$\alpha = 2$, $\gamma = -3$, $s=1$ in dimension $3$), the conjecture
is still open: the $L^\infty$ bound is missing (see however partial
results in this direction in \cite{luis-landau}), and the Schauder
estimates \cite{henderson2017c} do not cover this case even though
this last point is probably only a milder technical issue.

An important inspiration we draw from the case of the Landau equation
is that the iterative gain of regularity in the spirit of
\cite{henderson2017c} require a solution that decays, as
$|v| \to \infty$, faster than any algebraic power rate $|v|^{-q}$. We
expect the same general principle to apply to the Boltzmann
equation. 

\subsection{Main result}

%------------------------------------------------

We consider in this paper strong (classical) solutions to the
Boltzmann equation in the torus $x \in \mathbb T^d$ (periodic spatial
boundary conditions) with decay $O((1+|v|)^{-\infty})$,
i.e. polynomial of any order.
%------------------------------------------------------------------------------------
\begin{defn}[Classical solutions to the Boltzmann equation with rapid
  decay] \label{d:solutions-rapid-decay} Given $T \in (0,+\infty]$, we
  say that a function
  $f : [0,T] \times \mathbb T^d \times \R^d \to [0,+\infty)$ is a
  \emph{classical solution to the Boltzmann equation
    \eqref{eq:boltzmann} with rapid decay} if
\begin{itemize}
\item the function $f$ is differentiable in $t$ and $x$ and twice
  differentiable in $v$ everywhere;
\item the equation \eqref{eq:boltzmann} holds classically at every
  point in $[0,T] \times \mathbb T^d \times \R^d$;
\item for any $q>0$, $(1+|v|)^q f(t,x,v)$ is uniformly bounded on
  $[0,T] \times \mathbb T^d \times \R^d$.
\end{itemize}
\end{defn}
%------------------------------------------------------------------------------------

We chose the setting of classical solutions. This is natural because
we work under \emph{a priori} assumptions (the hydrodynamic bounds),
and moreover the only theory of existence of weak solutions available
in the case of long-range interactions is the theory of
``\emph{renormalized solution with defect measure}'' \cite{MR1857879},
that extends the notion of renormalized solutions of DiPerna and
P.-L. Lions \cite{MR1014927}, and these very weak solutions are too
weak to be handled by the methods of this paper. The rapid polynomial
decay we impose at large velocities is a qualitative assumption that
we make for technical reasons: just like the periodicity in $x$, it is
used to guarantee the existence of a first contact point in the
argument of maximum principle. It is specially needed in the case
$\gamma >0$. However the estimates in the conclusion of our theorem do
not depend on the decay rate as $|v| \to \infty$ that is initially
assumed for the solution (otherwise, the theorem would obviously be
empty). We discuss in Section~\ref{s:no-decay} how to relax this
qualitative assumption.

% -------------------------------------------------------------------
\begin{thm}[Pointwise moment bounds for the Boltzmann with hard or
  moderately soft potentials]\label{thm:upper}
  Let $\gamma \in (-2,2)$ and $s \in (0,1)$ satisfy
  $\gamma +2s \in [0,2]$ and $f$ be a solution of the Boltzmann
  equation~\eqref{eq:boltzmann} as in Definition
  \ref{d:solutions-rapid-decay} such that  $f(0,x,v)=f_0(x,v)$ in $\mathbb{T}^d \times \R^d$ and
\begin{equation}\label{eq:non-deg}
 \forall \, (t,x) \in [0,T] \times \mathbb T^d, \quad 0<m_0 \le M(t,x) \le M_0, \,
  E(t,x) \le E_0 \text{ and } H(t,x) \le H_0
\end{equation}
holds true for some positive constants $m_0,M_0,E_0,H_0$. It was then
proved in \cite{luis} that $f$ satisfies an $L^\infty$ a priori
estimate depending on these constants (see Theorem~\ref{thm:linfty}
recalled later); we establish here the following more precise decay
estimates at large velocities.
\begin{enumerate}
\item \label{i:3} \textbf{Propagation of ``pointwise moments'' for
    moderately soft and hard potentials.} There exists $q_0$ depending
  on $d$, $s$, $\gamma$, $m_0$, $M_0$, $E_0$, $H_0$ such that if
  $q \ge q_0$ and $f_0 \le C (1 + |v|)^{-q}$ for some $C>0$ then
  there exists a constant $\A$ depending on $C$, $m_0$, $M_0$, $E_0$,
  $H_0$, $q$, $d$, $\gamma$ and $s$, such that
  \[
    \forall \, t \in [0,T], \ x \in \mathbb{T}^d, \ v \in \R^d, \quad
    f (t,x,v) \le \A \left(1+|v|\right)^{-q}.
  \]
  
\item \label{i:2} \textbf{Appearance of ``pointwise moments'' for hard
    potentials.} If additionally $\gamma \in (0,2)$ then, for any
  $q>0$ there exists a constant $\A$ depending on
  $m_0,M_0,E_0, H_0, q, d, \gamma$ and $s$ and a power $\beta>0$
  depending on $d$, $q$, $\gamma$ and $s$ such that
  \[
    \forall \, t \in (0,T], \ x \in \mathbb{T}^d, \ v \in \R^d, \quad
    f(t,x,v) \le \A\left(1+t^{-\beta}\right) \left(1 + |v|\right)^{-q}.
  \]

\item \label{i:1} \textbf{Appearance of lower order ``pointwise
    moments'' for moderately soft potentials.} For all
  $\gamma \in (-2,0]$, there exists a constant $\A$ depending on
  $m_0,M_0,E_0, H_0, d, \gamma$ and $s$ such that
  \[
    \forall \, t \in (0,T], \ x \in \mathbb{T}^d, \ v \in \R^d, \quad
    f(t,x,v) \le \A\left(1+t^{-\frac{d}{2s}}\right) \left( 1 +
      |v|\right)^{-d-1-\frac{d \gamma}{2s}}.
  \]
\end{enumerate}
\end{thm}
%------------------------------------------------

\begin{remarks}
  \begin{enumerate}
  \item In the third point~\eqref{i:1}, the ``order'' of the pointwise
    decay is lower than what would ensure integrability in the energy
    bound $\int f |v|^2 \dd v <+\infty$ since
    $d+1+ \frac{\gamma d}{2s}<d+1$. Since $\gamma+2s \ge 0$
    (moderately soft potentials), $d+1+\frac{\gamma d}{2s} \ge
    1$. More precisely, in dimension $3$ and for an inverse power-law
    interaction force $Cr^{-\alpha}$ this is $(3\alpha-7)/2$ with
    $\alpha \in [3,5)$. However this bound is locally (in $v$)
    stronger than the energy bound as it is pointwise, and it does not
    depend on norms on derivatives through an interpolation argument.
  
  \item In the proof of point~\eqref{i:2}, our reasoning provides
    $\beta = d/(2s) + q/\gamma$ if $q$ is large, without claim of
    optimality.
  \item As discussed in Section~\ref{s:no-decay}, see 
    Theorem~\ref{thm:upper2}, the qualitative assumption of rapid decay
    can be relaxed entirely for~\eqref{i:1} and for~\eqref{i:2} for
    $q=d+1$ and also for~\eqref{i:3} when $\gamma \le 0$ and $q$ large
    enough. Finally for~\eqref{i:2} it can be weakened to $(1+|v|^{q_0})f$
    uniformly bounded on $t \in [0,T]$, $x \in \mathbb{T}^d$, $v \in
    \R^d$ for some $q_0$ large enough.
  \item It is conceivable that some versions points~\eqref{i:3}
    and~\eqref{i:2} of Theorem \ref{thm:upper} should hold in the
    cutoff case, probably with stronger conclusions. We are interested
    here in the non-cutoff model, so we have not investigated this
    problem. Note however that point~\eqref{i:1} is likely to be false
    in the cutoff case, i.e. to be of a strictly non-cutoff nature.
  \end{enumerate}
\end{remarks}

\subsection{Decay at large velocity in the Boltzmann theory}

The study of the decay at large velocity is central in the study of
solutions to the Boltzmann equation, and has a long history. Such
decay is necessary for instance to prove that appropriate weak
solutions satisfy the conservation of the kinetic energy (second
moment), and more generally appears in any regularity estimate.

\subsubsection{Moment estimates (weighted $L^1$ estimates)}
Measuring the decay at large velocity in terms of moments,
i.e. weighted integral $\int f |v|^q \dd v$, is a natural step in view
of the fact that the velocity space is unbounded and the collision
operator integrates over all velocities. The study of moments was
initiated for Maxwellian potentials ($\gamma=0$) in the spatially
homogeneous case in \cite{MR0075725,MR0075726}: closed systems of
exact differential equations are derived for polynomial moments and
their propagation in time is shown, without any possibility of
appearance. In the case of hard potentials ($\gamma>0$) with angular
cutoff (playing the role of ``$s=0$'') and spatial homogeneity (no $x$
dependency), the study of moments relies on the so-called
\emph{Povzner identities} \cite{MR0142362}:
\begin{itemize}
\item Elmroth \cite{MR684411} used them to prove that if any moment $q>2$
exists initially, then they remain bounded
for all times.
\item Desvillettes \cite{MR1233644} then showed that all moments are
  generated as soon as one moment of order $q>2$ exists initially.
\item Finally \cite{MR1697562,MR1716814} proved that even the
  condition on one moment of order $s>2$ can be dispensed with and
  only the conservation of the energy is required; it was later
  extended to the spatially homogeneous hard potentials without cutoff
  in \cite{MR1461113}.
\item Then Bobylev \cite{MR1478067}, through some clever refinement of
  the Povzner inequality and infinite summation, proved, for
  spatially homogeneous hard potentials with cutoff, the propagation
  of (integral) exponential tail estimates $\int f e^{C|v|^b} \dd v$
  with $b \in (0,2]$ and $C$ small enough if $b=2$.
\item This result was extended in \cite{MR2096050} to more general
  collision kernels, that remains variants of hard potentials with cutoff.
\item Finally the Bobylev's argument   was improved to obtain
  \emph{generation} of (integral) exponential tail estimates
  $\int f e^{C|v|^b}\dd v$ with $b \in (0,\gamma]$ in
  \cite{MR2264623,MR3005550}.
\item The case of measure-valued solutions in the spatially
  homogeneous hard potentials with cutoff is considered in
  \cite{MR2871802}, and the case of the Boltzmann-Nordheim equation
  for bosons was addressed in \cite{MR3493188}.
\end{itemize}

Let us also mention two important extensions of these methods:
\begin{itemize}
\item In the case of spatially homogeneous moderately soft potentials
  with cutoff, Desvillettes \cite{MR1233644} proved for
  $\gamma \in (-1,0)$ that initially bounded polynomial moments grow
  at most linearly with time and it is explained in
  \cite{villani-book} that the method applies to $\gamma \in
  [-2,0)$. This was later improved \cite{MR2359877} into bounds
  uniform in time thanks to the convergence to equilibrium.
\item In \cite[Section~5]{GMM}, the appearance and propagation of
  polynomial moments $L^1_v L^\infty_x(1+|v|^q)$ is proved for the the
  spatially inhomogeneous Boltzmann equation in $x \in \mathbb T^d$
  for hard spheres, as well as the appearance and propagation of
  exponential moments $L^1_v W^{3,1} _x (e^{c |v|})$. All these
  results assume bounds on the hydrodynamic quantities similar to what
  is assumed in this paper.
\end{itemize}

\subsubsection{Pointwise decay (weighted $L^\infty$ estimates)}

In the spatially homogeneous setting (with cutoff), the study of
pointwise decay goes back to Carleman~\cite{MR1555365,MR0098477},
where it was first studied for radially symmetric solutions
$f=f(t,|v|)$, and was further developed in \cite{MR711482} (see also
the $L^p$ bounds in \cite{MR816620,MR946968}). The first exponential
pointwise bound was obtained in \cite{gamba-panferov-villani-2009} and
the latter paper pioneered the use of the comparison principle for the
Boltzmann equation: the authors obtain pointwise Gaussian upper and
lower bounds. The method was extended in \cite{bobylev2017upper}, and
in \cite{gamba2017pointwise} (using estimates from
\cite{alonso2017exponentially}) where the authors prove exponential
(but not exactly Gaussian) upper bound for the space homogeneous
non-cutoff Boltzmann equation.

Regarding the spatially inhomogeneous setting, it is mentioned in
\cite[Chapter~2, Section~2]{villani-book} that: ``\textit{In the case
  of the full, spatially-inhomogeneous Boltzmann equation there is
  absolutely no clue of how to get such [moment] estimates. This would
  be a major breakthrough in the theory}''.  This refers to
unconditional moment bounds, and, as the result in
\cite[Section~5]{GMM} mentioned above shows, it is expected that some
of these estimates can be extended to the space inhomogeneous case
under the assumptions that the hydrodynamic quantities stay under
control. However, moment estimates obtained using Povzner inequalities
would, in the most optimistic scenario, involve an upper bound on a
weighted integral quantity with respect to $x$ and $v$. There seems to
be no natural procedure to imply pointwise upper bounds from
them. Indeed, in order to apply methods similar to
\cite{gamba-panferov-villani-2009,bobylev2017upper,gamba2017pointwise}
to the space inhomogeneous case, we would first need strong
$L^\infty_{t,x} L^1_v(\omega)$ moments for some weight $\omega$
(probably exponential). In this paper, we bypass any analysis of
moment estimates by obtaining pointwise upper bounds directly.

\subsection{Strategy of proof}

The proof of the main theorem consists in proving that the solution
satisfies $f(t,x,v) < g(t,v)$ with $g(t,v) := \A(t) (1 + |v|)^{-q}$,
for different choices of the function
$\A: (0,+\infty) \to (0,+\infty)$ and $q \in \mathbb (0,+\infty)$. The
appearance of pointwise bounds requires $\A(t) \to +\infty$ as
$t \to 0$, whereas the function $\A(t)$ is bounded near $t=0$ for
proving propagation of pointwise bounds. Without loss of generality,
it is convenient in order to simplify calculations to use instead the
barrier $g(t,v) = N(t) \min(1,|v|^{-q})$.

We then ensure that the comparison is satisfied initially and look for
the first time $t_0 >0$ when the inequality is invalidated. We prove
the existence of a first contact point $(t_0,x_0,v_0)$ such that
$f(t_0,x_0,v_0) = g(t_0,v_0)$, and search for a
contradiction at this point. The existence of this first contact point
follows from the rapid decay assumption in Definition
\ref{d:solutions-rapid-decay} and the compactness of the spatial domain. 

At the first contact point we have $f(t_0,x_0,v_0)= g(t_0,v_0)$.
Since the right hand side does not depend on $x$, we must have
$\nabla_x f(t_0,x_0,v_0)=0$. We also deduce that
$\partial_t f(t_0,x_0,v_0) \geq \partial_t g(t_0,v_0)$.  Therefore,
since $f$ solves the equation \eqref{eq:boltzmann}, we have
\begin{equation} \label{e:to-contradict} \partial_t g(t_0,v_0) \leq
  \partial_t f (t_0,x_0,v_0)= Q(f,f) (t_0,x_0,v_0)
\end{equation}
%\cyril{[I do not understand this part of the sentence:] using the local nature of the dependency in the second argument for the   negative term}. 
We then decompose the collision operator (using the
so-called \emph{Carleman representation}, see \eqref{e:split} and
\eqref{eq:decomposition} below) into
$Q=\mathcal G + \mathcal B + Q_{ns}$ where $\mathcal G$ is the
``good'' term, that is to say negative at large velocities,
$\mathcal B$ is the ``bad'' term, treated as a positive error term at
large velocities, and finally $Q_{ns}$ is a remaining ``non-singular / lower order
term''  where the angular singularity has been
removed by the so-called ``cancellation lemma''.  The core of the
proof then consists in proving that the ``good'' negative term
dominates over all the other terms at large velocities, hence yielding
a contradiction.

Remark that the only purpose of the rapid decay assumption in
Definition~\ref{d:solutions-rapid-decay} is to obtain this first
contact point $(t_0,x_0,v_0)$. In Section~\ref{s:no-decay} we explore
a setting in which we can relax this qualitative assumption: we add a
small correction term to the function $g(t,v)$ in order to ensure the
inequality $f < g$ for large values of $v$; we recover a large part of
Theorem \ref{thm:upper}, but run into technical problems when
$\gamma > 0$ (see Theorem~\ref{thm:upper2}). 

\subsection{Open questions}

Here are some natural questions that remain unanswered and are natural
problems to investigate in the future:
\begin{itemize}
\item Our result says, for some range of parameters, that the rate of
  decay of the solution $f(t,x,v)$ is faster than any power function
  $|v|^{-q}$ as $|v| \to \infty$. The most desirable result would be
  to obtain the appearance of exponential upper bounds on
  $\sup_{x,v} f e^{C|v|^\kappa}$ or the propagation of Gaussian upper
  bound $\sup_{x,v} f e^{C |v|^2}$ as in
  \cite{gamba-panferov-villani-2009} or \cite{css}. This seems to
  require new techniques.
\item Another open problem regards the range of parameters $\gamma,s$
  for which the bounds hold. This work is restricted to moderately
  soft potentials $\gamma+2s \in [0,2]$. The case
  $\gamma+2s \in (-1,0)$ (very soft potentials) is of great interest
  but seems out of reach with the current methods and requires new
  ideas. The non-physical range $\gamma+2s > 2$ presents a difficulty
  in that the energy estimate is insufficient to control the kernel
  $K_f$ defined in \eqref{e:kernel-K}.
\item It would also be interesting to relax the qualitative assumption
  of rapid decay in Definition~\ref{d:solutions-rapid-decay} of the
  solutions we use. We explore this question in
  Section~\ref{s:no-decay}. In the case $\gamma \leq 0$, we recover
  essentially the same result as in Theorem \ref{thm:upper} without
  assuming the rapid decay at infinity of solutions provided that
  $\gamma+2s < 1$. In the case $\gamma > 0$, we can always generate
  decay of the form $f \leq \A |v|^{-d-1}$. However, in order to
  obtain an upper bound that decays with a higher power, we need to
  make the qualitative assumption that
  $\lim_{|v| \to \infty} |v|^{q_0} f(t,x,v) = 0$ for some power $q_0$
  that depends on all the other parameters. We would naturally expect
  the estimates in Theorem \ref{thm:upper} to hold for solutions with
  only the energy decay.
\end{itemize}

\subsection{Organisation of the article} 
Section~\ref{s:prelim} reviews quickly results from previous works
that are used in the proof of the main result. The collision operator
is divided into different pieces which are estimated successively in
Section~\ref{s:estimates}. Section~\ref{s:proof} contains the proof of
the main result. Finally Section~\ref{s:no-decay} discusses how to
relax the assumption of rapid decay.

\subsection{Notation} For two real numbers $a,b \in \R$, we write
$a \wedge b$ for their minimum. Moreover, $a \lesssim b$ means that
$ a \le C b$ with $C$ only depending on dimension, $\gamma$, $s$ and
hydrodynamic quantities $m_0, M_0, E_0, H_0$. The notation
$a \lesssim_q b$ means that $C$ may additionally depend on the
parameter $q$.  Constants $C_q$, $R_q$ also depend on $q$, and can be
large. The constant $c_q$ is ``explicit'' in
Proposition~\ref{prop:G-bis}. We sometimes use the shorthand
$f'=f(v')$, $f'_*=f(v'_*)$, $f=f(v)$, $f_*=f(v_*)$. We will also
denote classically: $w:=v'-v$, $\omega = w/|w|$, $u:=v'_*-v$,
$\hat u = u/|u|$.

\section{Preliminaries results}
\label{s:prelim}

\subsection{Cancellation lemma and Carleman representation}

We split the Boltzmann collision operator in two, along an idea
introduced in \cite{advw}, and use the so-called \emph{cancellation
  lemma} from the latter paper to estimate the non-singular
part. Then, in the remaining singular part, we change variables to the
so-called \emph{Carleman representation} introduced in
\cite{MR1555365,MR0098477} (see also
\cite[Subsection~4.6]{villani-book} for a review, and developed in the
non-cutoff case in \cite[Section~4]{luis}). Given a velocity $v$, the
possible binary collisions can be parametrised (i) by $v_* \in \R^d$,
$\sigma \in \mathbb S^{d-1}$ which is sometimes called the
``$\sigma$-representation'' given in the previous section, (ii) by
$v_* \in \R^d$ and $\omega \in \mathbb S^{d-1}$ with
$\omega := (v'-v)/|v'-v|$ (see Figure~\ref{fig:collision}) which is
sometimes called the ``$\omega$-representation'', and (iii) by
$v' \in \R^d$ and $v'_* \in v + (v'-v)^\bot$ (the $(d-1)$-dimensional
hyperplan) which is called the ``Carleman representation'' to
acknowledge its introduction in \cite{MR1555365} in the radial
case. This alternative Carleman representation is used, as in previous
works, in order to write the Boltzmann collision operator as an
integral singular Markov generator applied to its second argument,
with an explicit kernel depending on its first argument.

The splitting is (without caring for now about convergences of integrals):
\begin{align}
\nonumber  & Q (f_1,f_2)(v) \\
\nonumber  & = \int_{\R^d \times \Sp} \Big[ f_1(v'_*) f_2(v') - f_1(v_*) f_2(v) \Big] B \dd v_* \dd \sigma \\
\nonumber  & = \int f_1(v'_*) \big[ f_2(v')-f(v) \big] B \dd v_*
    \dd \sigma + f_2(v) \int \big[ f_1(v'_*) - f_1(v_*)
               \big] B \dd v_* \dd \sigma \\
\label{e:split}  & =: Q_s(f_1,f_2) + Q_{ns}(f_1,f_2)
\end{align}
where ``s'' stands for ``singular'' and ``ns'' stands for
``non-singular''.

Let us first consider the non-singular part $Q_{ns}$. Given
$v \in \R^d$, the change of variables
$(v_*,\sigma) \mapsto (v'_*,\sigma)$ has Jacobian
${\rm d} v'_* {\rm d}\sigma = 2^{d-1} (\cos \theta/2)^2 \, {\rm d} v_*
{\rm d}\sigma$, which yields (same calculation as
\cite[Lemma~1]{advw})
\begin{equation*}
  Q_{ns}(f_1,f_2)(v)  = f_2(v) \int_{\R^d} \int_{\Sp} \big[ f_1(v'_*) - f_1(v_*)
            \big] B \dd v_* \dd \sigma 
               =: f_2(v) (f_1 * S)(v)
\end{equation*}
with
\begin{align*}
  S(u) & := \left|\mathbb S^{d-2}\right| \int_0 ^{\frac{\pi}{2}} (\sin
         \theta)^{d-2} \Big[ (\cos \theta/2)^{-d}
         B\left( \frac{|u|}{\cos
         \theta/2}, \cos \theta \right) - B(|u|,\cos \theta) \Big]
         \dd \theta\\
       & = \left|\mathbb S^{d-2}\right| |u|^\gamma
         \int_0 ^{\frac{\pi}{2}} (\sin
         \theta)^{d-2} \Big[ (\cos \theta/2)^{-d-\gamma}
         - 1 \Big]  b(\cos \theta)
         \dd \theta \\
       & =: C_S |u|^\gamma
\end{align*}
where we have used the precise form~\eqref{assum:B} of the collision
kernel in the second line. The constant $C_S >0$ is finite and only
depends on $b$, $d$, and $\gamma$. In short, the cancellation lemma is
a kind of discrete integration by parts where the singularity of the
fractional derivative is pushed onto the kernel itself. 

Let us consider the singular part $Q_s$. We change variables (Carleman
representation) according to $(v_*,\sigma) \mapsto (v',v'_*)$ as
described above. The Jacobian is
${\rm d} v_* {\rm d}\sigma  = 2^{d-1} |v-v'|^{-1} |v-v_*|^{-(d-2)} \, {\rm d} v'
{\rm d} v'_*$ (see for instance \cite[Lemma~A.1]{luis}): 
\begin{equation}\label{eq:Carl-rep}
  Q_s(f_1,f_2)(v) = \PV \int_{\R^d} K_{f_1}(v,v') \big[ f_2(v')-f_2(v) \big] \dd v',
\end{equation}
where
\begin{align} \nonumber
  K_{f_1}(v,v') & := \frac{1}{|v'-v|} \int_{v'_* \in \; v +
              (v'-v)^\bot} f_1(v'_*) |v-v_*|^{\gamma-(d-2)} b(\cos \theta) \dd
              v'_* \\  \label{e:kernel-K}
            & := \frac{1}{|v'-v|^{d+2s}} \int_{v'_* \in \; v +
              (v'-v)^\bot} f_1(v'_*) |v-v'_*|^{\gamma+2s+1} \tilde b(\cos \theta) \dd
              v'_* 
\end{align}
where we have used the assumption \eqref{assum:B} and in particular
the fact that $b(\cos \theta) = |v-v'|^{-(d-1)-2s}
|v-v_*|^{(d-2)-\gamma} |v-v'_*|^{\gamma+2s+1} \tilde b(\cos
\theta)$. The notation $\PV$ denotes the Cauchy principal value around
the point $v$. Note that it is needed only when $s \ge [1/2,1)$. 

One can reverse the order of the integration variables to get the
alternative formula
\begin{equation}\label{eq:Carl-rep-reverse}
  Q_s(f_1,f_2) = \PV \int_{v'_* \in \R^d} f_1(v'_*) |v-v'_*|^{\gamma+2s}
  \int_{v'\in v +
    (v-v'_*)^\bot} \frac{\big[ f_2(v')-f_2(v) \big] \tilde b}{|v'-v|^{d-1+2s}} \dd v'.
\end{equation}
Note that we have used the following standard manipulation:
  \begin{align*}
    & \int_{u \in \R^d} \int_{w \bot u} F(u,w) \dd w \dd u = \int_{u,w\in \R^d}
      F(u,w) \delta(w \cdot \hat u) \dd w \dd u \\
    & = \int_{u,w\in \R^d} F(u,w) \delta(|w| (\hat w \cdot \hat u))
      \dd w \dd u
      =  \int_{u,w\in \R^d} F(u,w) |w|^{-1} \delta(\hat w \cdot \hat
      u) \dd w \dd u\\
    & = \int_{u,w\in \R^d} F(u,w) \frac{|u|}{|w|}  \delta(u \cdot \hat
      w) \dd u \dd w = \int_{w \in \R^d} \int_{u \bot w}
      \frac{|u|}{|w|} F(u,w) \dd u \dd w.
  \end{align*}
  
\subsection{Lower bound induced by the hydrodynamic bounds}

It is classical that, for each $x$, the controls on the local mass,
energy and entropy, and the non-vacuum condition, together imply that
the mass is bounded below and cannot concentrate in a zero-measure
set: hence it implies pointwise lower bound on non zero-measure sets.
%-------------------------------------------------------------------------
\begin{lemma}[Lower bound on a set with positive
  measure] \label{lem:D} Under the assumption \eqref{eq:non-deg},
  there exists $R_0>0$ such that for all $t \in [0,T]$ and
  $x \in \mathbb{T}^d$ there exists a set
  $\mathcal{D} = \mathcal{D}(t,x) \subset B_{R_0}$ such that
  \[ \forall \, v \in \mathcal{D}(t,x), \; f(t,x,v) \ge c_0 \quad \text{ and } \quad
  |\mathcal{D}(t,x)| \ge \mu >0 \]
  for $R_0$ and $\mu$ only depending on $M_0,m_0,E_0,H_0$ and
  dimension.
\end{lemma}

\begin{proof}
  The proof is elementary and can be found for instance in
  \cite[Lemma~4.6]{luis}. It follows from the classical fact that the
  entropy bound implies the non-concentration estimate
\[
  \int_A f(t,x,v) \dd v \lesssim_{M_0,H_0} \varphi(|A|) \quad \mbox{
    with } \varphi(r) = \ln(1+r) + \left[ \ln\left(r^{-1}\right) \right]^{-1}
\]
and $A$ a Borel set and $|A|$ its Lebesgue measure. The energy bound
provides tightness and prevents the mass from being arbitrarily far
from the origin.  
\end{proof}

\subsection{The cone of non-degeneracy}

%----------------------------------------------------------------------------
We recall from \cite{luis,is} the following more subtle result. 
%-------------------------------------------------------------------------
\begin{lemma}[Cone of non-degeneracy]\label{lem:cone-nd}
  Consider a non-negative function $f$
  satisfying~\eqref{eq:non-deg}. Then there are constants
  $c_0,C_0,\mu, \mu'>0$ (depending on $d$, $\gamma$, $s$, $m_0$,
  $M_0$, $E_0$ and $H_0$) such that for any $t \in [0,T]$,
  $x \in \mathbb T^d$ and $v \in \R^d$, there exists a cone of
  directions $\Xi = \Xi(t,x,v)$ that is symmetric
  (i.e. $\Xi(t,x,v) = -\Xi(t,x,v)$) and so that
   \[
    \forall \, v \in \R^d, \ \forall \, v' \in v+\Xi(t,x,v), \quad K_f(v,v') \geq c_0 (1+|v|)^{1+\angulars+\gamma}
    |v-v'|^{-d-\angulars}
  \]
  and 
  \[
    \Xi(t,x,v) \subset \left\{ w \in \R^d \ | \ |v \cdot w| \le C_0
      |w|\right\}
  \]
  and for any $r >0$
  \begin{equation}\label{eq:cone-estim}
     \frac{\mu r^d}{(1+|v|)} \le \left|\Xi(t,x,v) \cap B_r\right| \le
    \frac{\mu' r^d}{(1+|v|)}.
  \end{equation}
\end{lemma}

\begin{proof}
  The proof is gathered from \cite[Lemma~4.8]{luis} and
  \cite[Lemma~7.1]{luis}.  Observe that by changing the order of
  integration, for any integrable function $F$ on $\R^d$:
  \[
    \int_{\omega \in \mathbb{S}^{d-1}} \int_{\{u \in \R^d, \ u \bot
      \omega\}} F(u) \dd u \dd \omega = \left| \mathbb{S}^{d-2}\right|
    \int_{\R^d} \frac{F(u)}{|u|} \dd u.
  \]
  We distinguish the two cases $|v| \le 2 R_0$ and $|v| > 2 R_0$,
  where $R_0$ is the upper bound on elements of $\mathcal{D}(t,x)$ in
  Lemma~\ref{lem:D}.
  \medskip

  \noindent
  {\em Case when $|v| \le 2 R_0$.} We estimate
  \begin{align*}
    \int_{\omega \in \mathbb{S}^{d-1}} \int_{\{ u \in \R^{d}, \ u \bot \omega \}}
    \one_{\mathcal{D}(t,x)}(v+u) \dd u \dd \omega 
    & = \left| \mathbb{S}^{d-2}
      \right| \int_{\R^d} \one_{\mathcal{D}(t,x)}(v+u) |u|^{-1} \dd u \\
    & = \left| \mathbb{S}^{d-2}
      \right| \int_{v'_* \in \mathcal{D}(t,x)} |v'_*-v|^{-1} \dd v'_*
  \end{align*}
  which is bounded below by some positive constant $\delta_0>0$
  independent of $t \in \R$, $x \in \mathbb{T}^{d-1}$,
  $v \in B_{2R_0}$. Because of the upper bound
  \begin{equation}\label{eq:upper-cone}
    \int_{\{ u \in \R^{d}, \ u \bot \omega \}}
    \one_{\mathcal{D}(t,x)}(v+u) \dd u \le
    \int_{\{ u \in \R^{d}, \ |u| \le R_0 \}}
    \dd u \lesssim_{R_0} 1
  \end{equation}
  following from the boundedness of $\mathcal{D}(t,x)$, we deduce that
  there exists $\mu_0,\lambda_0>0$ such that for all $t \in \R$,
  $x \in \mathbb{T}^{d-1}$ and $v \in B_{2R_0}$, there exists a set
  $\Xi(t,x,v) \cap \mathbb{S}^{d-1}$ of unit vectors $\omega$ such
  that $| \Xi(t,x,v) \cap \mathbb{S}^{d-1} | \ge \mu_0$ and
  \[
    \forall \, \omega \in \Xi(t,x,v) \cap \mathbb{S}^{d-1}, \quad
    \int_{\{ u \in \R^{d}, \ u \bot \omega \}}
    \one_{\mathcal{D}(t,x)}(v+u) \dd u \dd \omega \ge \lambda_0
  \]
  for some $\lambda_0>0$. Since the integrand above is even as a
  function of $\omega$, the cone $\Xi(t,x,v)$ can be chosen symmetric.
  \medskip
  
  \noindent
  {\em Case when $|v| > 2 R_0$.} We estimate
  \begin{align*}
    \int_{\omega \in \mathbb{S}^{d-1}} \int_{\{ u \in \R^{d}, \ u \bot
    \omega \}} \one_{\mathcal{D}(t,x)}(v+u) \dd u \dd \omega
    &=\left| \mathbb{S}^{d-2} \right| \int_{v'_* \in \mathcal{D}(t,x)}
      |v'_*-v|^{-1} \dd v'_*\\
    & \ge \frac{\delta_1}{1+|v|} 
  \end{align*}
  for some positive constant $\delta_1 >0$ independent of $t \in \R$,
  $x \in \mathbb{T}^{d-1}$, $v \in B_{2R_0}^c$. Given the upper
  bound~\eqref{eq:upper-cone} on the one hand and the fact that the
  support of the function
  \[
    \omega \mapsto \int_{\{ u \in \R^{d}, \ u \bot \omega \}}
    \one_{\mathcal{D}(t,x)}(v+u) \dd u
  \]
  is included in the set $\{ \omega \in \mathbb{S}^{d-1} \, : \,
  |\omega \cdot v| \le C\}$ for some constant $C>0$, we deduce that
  there exists $\mu_1,\lambda_1>0$ such that for all $t \in \R$,
  $x \in \mathbb{T}^{d-1}$ and $v \in B_{2R_0}$, there exists a set
  $\Xi(t,x,v) \cap \mathbb{S}^{d-1}$ of unit vectors $\omega$ such
  that
  \[
    \left| \Xi(t,x,v) \cap \mathbb{S}^{d-1} \right| \ge \frac{\mu_1}{1+|v|}
  \]
  and
  \[
    \forall \, \omega \in \Xi(t,x,v) \cap \mathbb{S}^{d-1}, \quad
    \int_{\{ u \in \R^{d}, \ u \bot \omega \}}
    \one_{\mathcal{D}(t,x)}(v+u) \dd u \dd \omega \ge \lambda_1.
  \]
  Since the integrand above is even as a function of $\omega$, the
  cone $\Xi(t,x,v)$ can be chosen symmetric. It lies by construction
  in the equatorial region required. 

  The cone $\Xi(t,x,v)$ built above satisfies the statement.  
\end{proof}

\section{Technical estimates on the collision operator}
\label{s:estimates}

We consider a contact point where
\begin{equation}\label{eq:contact}
  f (t,x,v') \le g (t,v') \text{ for all } v' \in \R^d \quad \text{ and } \quad f(t,x,v) = g(t,v).
\end{equation}
Since the collision operator does not act on the $t$ and $x$
variables, we omit them in most of this section to keep calculations
uncluttered. %We however keep track of the dependencies on these variables.

\subsection{Estimates of the collision operator at the first contact point}

In order to estimate the singular part $Q_s(f,f)$ of the collision
operator, we split it into a ``good'' term, negative at large
velocities, and a ``bad'' term, treated as a smaller error at large
velocities: define $c_1(q) = q^{-1}/20$ and 
\begin{equation}\label{eq:decomposition}
  Q_s(f,f) = \mathcal{G} (f,f) + \mathcal{B}(f,f)
\end{equation}
with 
\begin{align*}
  \mathcal{G} (f_1,f_2) & = \PV \int_{|v'_*| \le c_1(q) |v|} f_1(v'_*)
                      |v-v'_*|^{\gamma+2s} \int_{v' \in v +
                      (v-v'_*)^\bot}  
                      \frac{\big[ f_2(v')-f_2(v) \big] \tilde b}{|v-v'|^{d-1+2s}} \dd v' \dd v'_*   \\
  \mathcal{B}(f_1,f_2) & =  \PV \int_{|v'_*| > c_1(q) |v|} f_1(v'_*)
                     |v-v'_*|^{\gamma+2s} \int_{v' \in v +
                     (v-v'_*)^\bot} \frac{\big[ f_2(v')-f_2(v) \big] \tilde b}{|v-v'|^{d-1+2s}}  \dd v' \dd v'_*. 
\end{align*}

Note first that this decomposition is based on the
representation~\eqref{eq:Carl-rep-reverse} but the order of
integration will sometimes be reversed back to the
representation~\eqref{eq:Carl-rep}, depending on technical
conveniency. Note second that the idea behind this decomposition is to
isolate the ``good'' configurations when $v'_*$ is close enough to
zero, where the bulk of the mass is located. Note finally that under
assumption~\eqref{eq:contact} on $f$ one has $Q(f,f) \le Q(f,g)$ and
similarly $Q_s(f,f) \le Q_s(f,g)$ and
$\mathcal{G}(f,f) \le \mathcal{G}(f,g)$ and
$\mathcal{B}(f,f) \le \mathcal{B}(f,g)$. We bound from below
$\mathcal{G}(f,g)$ for large $q$ and $\mathcal{G}(f,f)$ for
not-so-large $q$, and we bound from above $\mathcal{B}(f,f)$ and
$Q_{ns}(f,f)$ successively in the next subsections.

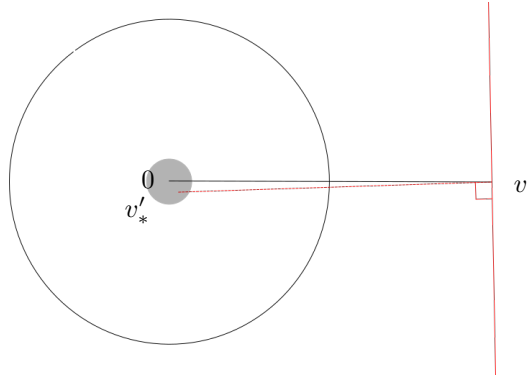
\begin{figure}
\setlength{\unitlength}{1.1cm} 
\begin{picture}(4.5,5)
\put(-1,0){\includegraphics[height=4.95cm]{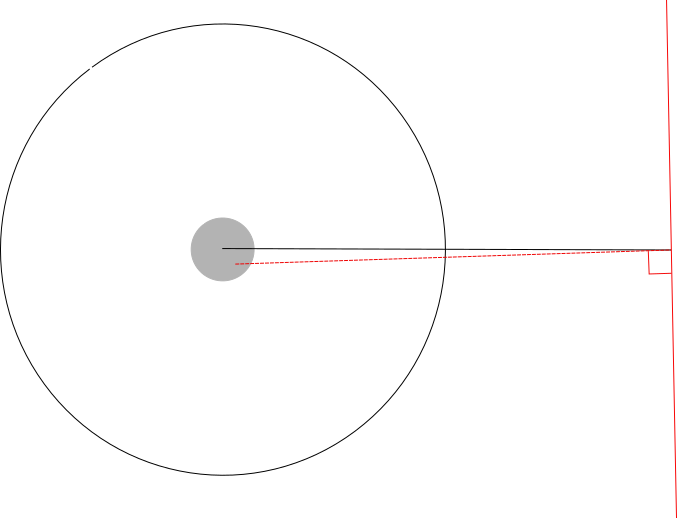}}
\put(5.1,2.2){$v$}
\put(0.4,1.9){$v'_*$}
\put(.6,2.25){$0$}
\end{picture}
\caption{The good term: $\mathcal G$ corresponds to the
  integration over the red (plain) line
  ($\{ v' \in \R^d : v' -v \perp w\}$). The grey ball is of radius
  $c_q |v|$ while the larger ball is of radius $|v|/2$.}
\label{fig:good}
\end{figure}

\subsection{Lower bound on the good term $\mathcal G$ for large $q$}

%-------------------------------------------------
\begin{prop}[Estimate of {$\mathcal{G} (f,g)$} useful for large
  $q$] \label{prop:G-bis} Let $f$ be a non-negative function satisfying \eqref{eq:non-deg} and
  $g = \A \min(1,|v|^{-q})$, $q \ge 0$. Then there exists a radius
  $R_q \geq 1$ so that
  \[ \forall \, v \in \R^d \ | \ |v| \geq R_q, \quad
    \mathcal{G}(f,g)(v) \lesssim - (1+q)^s |v|^{\gamma} g(v). \] 
%where  the constants $C_R >0$ and in the latter inequality are independent of  $q$.
\end{prop}
%----------------------------------------------------------------------
We first estimate from above the inner integral in the following lemma.
\begin{lemma} \label{l:good-cone} Let $q \geq 0$ and
  $g (v) = \A \min(1,|v|^{-q})$. Then for all $v \in \R^d$ such that
  $|v| \ge 2$ and all $v'_* \in \R^d$ such that $|v'_*| < c_1(q)|v|$,
  we have (with a constant uniform in $q \ge 0$)
\[ 
  \int_{v'\in \; v + (v-v'_*)^\bot} \big[ g(v') - g(v)\big] \frac{\tilde
    b(\cos \theta)}{|v'-v|^{d-1+2s}} \dd v' \lesssim - (1+q)^s \A
  |v|^{-2s-q}.
\]
\end{lemma}
%--------------------
\begin{proof}
  We first prove that this integral is non-positive when $(v'-v)$ is
  small enough with respect to $v$. We then prove that when $(v'-v)$
  is large enough in proportion to $v$ then $|v'|$ is larger than
  $|v|$ (in this step we use the assumption $|v'_*| < c_1(q) |v|$) which
  gives an explicit negative upper bound due to the decay of $g$. The
  geometric interpretation is simple: when $v'_*$ is small with
  respect to $v$ and $v$ is large, then the cone of possible
  directions for $(v'-v)$ is close to orthogonal to $v$, and when
  $(v'-v)$ is not too small $v'$ leaves $B(0,|v|)$. 

  Define $c_2(q) := (1+q)^{-1/2}/20$ (note the different asymptotic
  behaviour as compared to $c_1(q)$) and assume first that
  $|v-v'| < 4 c_2(q) |v|$. Then $|v'| \ge (1-4c_2(q))|v| \ge 1$ since
  $|v| \ge 2$, and $g(v') = N |v'|^{-q}$. By assumption $|v| \ge 1$
  and thus $g(v) = N |v|^{-q}$. The integration kernel is invariant
  under rotation around the axis $(vv*)$ and therefore
  \begin{align*}
    & \int_{v'\in \; v + (v-v'_*)^\bot} \big[ g(v') - g(v)\big] \frac{\tilde
    b(\cos \theta)}{|v'-v|^{d-1+2s}} \dd v' \\ =
    & \int_{v'\in \; v + (v-v'_*)^\bot} \Big[ g(v') - g(v) - \nabla g(v)
    \cdot (v'-v) \Big] \frac{\tilde
    b(\cos \theta)}{|v'-v|^{d-1+2s}} \dd v'.
  \end{align*}

  Taylor expand the integrand: there is some $\theta \in (0,1)$ and
  $v'_\theta := v + \theta(v'-v)$ so that
\begin{align}
  \nonumber
  \Big[ g(v')-g(v) - \nabla g (v) \cdot (v'-v) \Big]
  &=  \frac12 D^2 g (v'_\theta) (v'-v)\cdot (v'-v) \\
\label{eq:taylor} 
  & = \frac{\A}{2} |v'_\theta |^{-q-2} q
    \left(\frac{(q+2)}{|v'_\theta|^2} |v'_\theta \cdot (v'-v)|^2 - |v'-v|^2 \right).
\end{align}
Since $(v-v') \perp (v-v'_*)$ and $|v-v'| \le 4 c_2(q) |v|$ and $|v'_*|\le c_1(q) |v|$,  we have 
\begin{align*}
  |v'_\theta \cdot (v'-v)| & \le |v\cdot (v'-v)| + |v'-v|^2 = |v'_*
  \cdot (v'-v)| + |v'-v|^2 \\ & \le c_1(q) |v||v'-v| + |v'-v|^2 \le
                                \left( c_1(q) + 4 c_2(q) \right) |v| |v'-v|
\end{align*}
and $|v'_\theta| \geq (1-4c_2(q)) |v|$. We deduce
\begin{equation} \label{e:small z}
  \Big[ g(v')-g(v) - \nabla g (v) \cdot (v'-v) \Big] \le \frac{q
    \A}{2|v'_\theta |^{q+2}}
\left((q+2)\left(\frac{c_1(q) + 4 c_2(q)}{1-4c_2(q)}\right)^2-1 \right)|v'-v|^2 \le 0
\end{equation}
since
\[
  \left(\frac{c_1(q)+4c_2(q)}{1-4c_2(q)}\right)^2 \le \frac{1}{9(q+1)} \quad \Rightarrow
  \quad (q+2) \left(\frac{5c_q}{1-4c_q}\right)^2 \le \frac{1}{3}
\]
uniformly for $q \ge 0$, due to the smallness assumptions on
$c_1(q) \le c_2(q)$.

When $|v'-v| \ge 4 c_2(q)|v|$, then (using the smallness and orthogonality
properties as before):
\begin{align*}
|v'|^2 & = |v|^2 + |v'-v|^2 + 2 v'\cdot (v'-v) = |v|^2 + |v'-v|^2 + 2
         v'_* \cdot (v'-v)\\
       & \ge |v|^2 + |v'-v|^2 -2c_q |v||v'-v| \ge |v|^2 + |v'-v|
         \left( |v'-v| -2c_1(q) |v| \right)   \\
       & \ge |v|^2 + |v'-v| 2c_2(q) |v| \ge (1 + 8 c_2(q)^2) |v|^2 
\end{align*}
where we used $c_1(q) \le c_2(q)$, and in particular, 
\begin{equation}\label{e:large z}
g(v') - g(v)  \le - \A \left[1-\left(1+8c_q^2\right)^{-\frac{q}2} \right]|v|^{-q} \lesssim -\A |v|^{-q}.
\end{equation}
The last inequality uses $1 - (1+8 c_2(q)^2)^{-\frac{q}2} \to 1 - e^{-4/19^2} >0$
as $q \to \infty$. (This is where we use that $c_2(q) = O(q^{-1/2})$
rather than $Q(q^{-1})$ like $c_1(q)$.)

We deduce from \eqref{e:small z} and \eqref{e:large z} that, if
$|v'_*| \le c_1(q) |v|$ and $|v| \ge 2$ then
\begin{align*}
  & \int_{v' \in \; v + (v-v'_*)^\bot} \big[ g(v') - g(v) \big] \tilde
    b(\cos \theta) \frac{\dd v'}{|v'-v|^{d-1+2s}} \\ 
  & \le \int_{v' \in \; v + (v-v'_*)^\bot} \big[ g(v') - g(v) \big] \tilde
    b(\cos \theta) \chi_{\{|v-v'| \ge 4c_2(q) |v|\}} \frac{\dd v'}{|v'-v|^{d-1+2s}}  \\
  & \lesssim - \A |v|^{-q} \int_{v' \in \; v + (v-v'_*)^\bot} \tilde
    b(\cos \theta) \chi_{\{|v-v'| \ge 4c_2(q) |v|\}} \frac{\dd v'}{|v'-v|^{d-1+2s}} \\
  & \lesssim - \A c_2(q)^{-2s}|v|^{-q-2s} \approx -q^s \A |v|^{-q-2s}.
\end{align*}
This achieves the proof of the lemma. 
\end{proof}
We can now prove Proposition~\ref{prop:G-bis}. 
%-------------------------------------------------
\begin{proof}[Proof of Proposition~\ref{prop:G-bis}]
We estimate $\mathcal G(f,g)$ using Lemma~\ref{l:good-cone}.
\begin{align*}
   \mathcal{G} (f,g)  & = \PV \int_{|v'_*| \le c_1(q) |v|} f(v'_*)
                      |v-v'_*|^{\gamma+2s} \int_{v' \in v +
                      (v-v'_*)^\bot}  
                      \frac{\big[ g(v')-g(v) \big] \tilde b}{|v-v'|^{d-1+2s}}
                      \dd v' \dd v'_*   \\
&\lesssim - q^{s} \A |v|^{-q-2s} \int_{|v'_*| \le c_1(q) |v|} f(v'_*) |v-v'_*|^{\gamma+2s} \dd v'_*, \\
&\lesssim - q^{s} \A |v|^{-q+\gamma}.
\end{align*}
The last inequality follows from the lower bound on $\tilde b$ and
choosing $R_q \ge R_0 c_1(q) ^{-1}$, where $R_0$ is the radius of Lemma
\ref{lem:D}: then the lower bound of Lemma~\ref{lem:D} implies a lower
bound on
\[
  \int_{|v'_*| \le c_1(q) |v|} f(v'_*) |v-v'_*|^{\gamma+2s} \dd v'_* \gtrsim
  |v|^{\gamma+2s} \int_{|v'_*| \le R_0} f(v'_*) \dd v'_* \gtrsim |v|^{\gamma+2s}
\]
since $|v-v'_*| \gtrsim |v|$ on the domain of integration (and given
$|v| \ge R_q$).
\end{proof}

\subsection{Lower bound on the good term $\mathcal{G}$ for
  not-so-large $q$}

The coercivity constant $(1+q)^s$ in the previous estimate is not
large enough to dominate other bad and non-singular terms for
not-so-large $q$: we therefore prove a second estimate inspired by the
study of the $L^\infty$ norm in \cite{luis}.
%-------------------------------------------------
\begin{prop}[Estimate of {$\mathcal{G} (f,f)$} for not-so-large
  $q$] \label{prop:G} Assume $f$ satisfies \eqref{eq:contact} for
  $g = \A \min(1,|v|^{-q})$ with $q \ge 0$. Then there exists
  $R_q \geq 1$ so that
  \[ \forall \, v \in \R^d \ | \ |v| \ge R_q, \quad \mathcal{G}
    (f,f)(v) \lesssim_q - g(v)^{1+\frac{2s}d}
    |v|^{\gamma+2s+\frac{2s}d}.
  \] 
\end{prop}
\begin{remark}
  Note that the constant here depends on $q$, but we do not track this
  dependency since this proposition will be used for not-so-large
  values $q \in [0,d+1]$.
\end{remark}
%-------------------------------------------------
\begin{proof}
  We first claim that the estimate this proposition is implied by the
  previous Proposition~\ref{prop:G-bis} whenever $|v| \ge 1$ and
  $\A |v|^{-q} \lesssim_q |v|^{-(d+1)}$. Indeed for such choice of
  $q$ and $v$ one has
  \[
    q^s |v|^\gamma g(v) \gtrsim_q g(v)^{1+\frac{2s}d}
    |v|^{\gamma+2s+\frac{2s}d}.
  \]  

  Consider now the case where $g(v) \ge C_q |v|^{-(d+1)}$ for a
  constant $C_q>0$ large enough (depending on $q$) to be chosen
  later. We proceed as in the proof of Proposition~\ref{prop:G-bis},
  but refine it in that we estimate the difference $f(v') - g(v')$:
\begin{align*}
  &  \mathcal G(f,f)(v) \\
  &\le \PV \int_{|v'_*| \le c_1(q) |v|} f(v'_*) |v-v'_*|^{\gamma+2s}
    \left\{ \int_{v' \in \; v + (v-v'_*)^\bot} \frac{\left[f(v')-g(v)
    \right] \tilde b}{|v-v'|^{d-1+2s}} \dd v'  \right\}  \dd
    v'_*,  \\
  &= \PV \int_{|v'_*| \le c_1(q) |v|} f(v'_*) |v-v'_*|^{\gamma+2s} \bigg\{
    \int_{v' \in \; v+ (v-v'_*)^\bot} \frac{\left[g(v')-g(v)\right]
    \tilde b}{|v-v'|^{d-1+2s}} \dd v' \\ 
  & \qquad \qquad \qquad \qquad \qquad \qquad \qquad \qquad + \int_{v'
    \in \; v+(v-v'_*)^\bot}
    \frac{\left[f(v')-g(v')\right] \tilde b}{|v-v'|^{d-1+2s}} \dd v' \bigg\}  \dd v'_*
\end{align*}
(note that second principal value is well-defined since $f(v)=g(v)$). 

The first of the two inner integral terms is negative because of Lemma
\ref{l:good-cone}. Thus
\begin{align*}
  & \mathcal G(f,f)(v) \\ 
  & \leq \PV \int_{|v'_*| \le c_1(q) |v|} f(v'_*) |v-v'_*|^{\gamma+2s}
    \bigg\{ \int_{v' \in \; v+(v-v'_*)^\bot}
    \frac{\left[f(v')-g(v')\right]\tilde b}{|v'-v|^{d-1+2s}} \dd
    v'\bigg\}  \dd v'_*,\\
  &\le \PV  \int_{\R^d} \left[f(v')-g(v')\right] K_{\bar f}(v,v') \dd
    v' \le 0.
\end{align*}
where we have exchanged the order of integration and where
$K_{\bar f}$ denotes the kernel~\eqref{e:kernel-K} with the truncated
$\bar f(v'_*) := f(v'_*) \one_{|v'_*|\le c_1(q) |v|}$. If $|v|$ is
sufficiently large, the estimates in Lemma~\ref{lem:cone-nd} hold for
$K_{\bar f}$ as well since $f$ and $\bar f$ share comparable bounds on
their hydrodynamic quantities.

Let us estimate the measure of points $w \in \Xi(t,x,v)$, the cone
from Lemma \ref{lem:cone-nd}, such that $f(v+w) \ge g(v)/2$. Note that
for sufficiently large $|v|$, whenever $w \in \Xi(t,x,v)$, the
almost-orthogonality condition in Lemma~\ref{lem:cone-nd} implies
\[
  |v+w|^2 \geq |v|^2 + |w|^2 - 2 C_0|w| \geq \frac{|v|^2}{2}
\]
and therefore
\[
  \left|\left\{ w \in \Xi(t,x,v) : f(v+w) \ge \frac{g(v)}{2} \right\}\right|
  \le \frac2{g(v)} \int_{w \in \Xi(t,x,v)} f(v+w) \dd w \le \frac{4E_0}{|v|^2
    g(v)}.
\]

The estimate~\eqref{eq:cone-estim} from Lemma \ref{lem:cone-nd}
implies that we can pick $r>0$ such that
\[
  \left|\Xi(t,x,v) \cap B_r \right| = \frac{4^2E_0}{|v|^2 g(v)}.
\]
The corresponding $r$ is given by 
$r \approx \left(|v|^{-1} g(v)^{-1}\right)^{\frac1d}$ and for this
choice of $r$ we have
\[
  \left|\Xi(t,x,v) \cap B_r \right| \ge 4 \left|\left\{ w \in \Xi: f(v+w)
      \ge \frac{g(v)}{2} \right\}\right|.
\]
This implies that three fourth of the $w \in \Xi(t,x,v) \cap B_r$ satisfy
$f(v+w) \le g(v)/2$.

Going back to our estimate on $\mathcal G$, we restrict the domain of
integration (since the integral is non-positive)
\[
\mathcal G(f,f)(v) \leq \int_{\Xi(t,x,v) \cap B_r \cap \left\{f(v+w) \le \frac{g(v)}{2}\right\}}
\left[ \frac{g(v)}{2} - g(v+w) \right] K_{\bar f}(v,v') \dd v'.
\]

This is a useful estimate when $g(v+w) > g(v)/2$ with
$w \in \Xi(t,x,v) \cap B_r$. Recall that we assume that
$g(v) \gtrsim C_q |v|^{-d-1}$ for an arbitrarily large constant
$C_q$. Let us pick $C_q$ large so that if
$g(v) \gtrsim C_q |v|^{-d-1}$ then
\[
  r \lesssim \left(|v|^{-1} g(v)^{-1}\right)^{\frac1d} \leq \left[ 1-
    \left( \frac{3}{4} \right)^{1/q} \right] |v|,
\]
so that $g(v+w) \geq \frac{3}{4} g(v)$ for $w \in B_r$. Note that the
latter inequality is always satisfied when $q=0$ without
extra-condition on $v$, and for large $q$ the constant $C_q=O(q)$. 

Therefore, we get
\begin{align*} 
  \mathcal G(f,f)
  &\leq -\frac{g(v)} 4 \int_{\Xi(t,x,v) \cap B_r \cap \left\{f(v+w) \le \frac{g(v)}{2}\right\}} K_{\bar f}(v,v') \dd v', \\
  &\lesssim -g(v) |v|^{\gamma+2s+1} r^{-d-2s} \left|\Xi(t,x,v) \cap B_r \cap
    \left\{f(v+w) \le \frac{g(v)}{2}\right\}\right|
\end{align*} 
where we have used the estimate on the kernel of
Lemma~\ref{lem:cone-nd}. We use now
\[
  \left|\Xi(t,x,v) \cap B_r \cap \left\{f(v+w) \le \frac{g(v)}{2} \right\}\right| \ge \frac34 |\Xi(t,x,v) \cap
  B_r| \approx r^d |v|^{-1}
\]
that follows from our choice of $r$ to deduce
\begin{align*}
  \mathcal G(f,f)(v) \lesssim -g(v) |v|^{\gamma+2s} r^{-2s} =
  - g(v)^{1+2s/d} |v|^{\gamma+2s+2s/d}
\end{align*}
which concludes the proof.
\end{proof}

\begin{remark} Here, we interpret in terms of the collision process on
  $v,v',v_*,v'_*$ the two last estimates for the good term.

  The first estimate given by Proposition~\ref{prop:G-bis} is
  generated by the angles $\theta$ such that
  \[
    |\sin (\theta/2)| = \frac{|v'-v|}{|v'_*-v|} \ge
    \frac{c_2(q)}{1+c_1(q)} = O(q^{-\frac12}).
  \]
  Hence, in some sense, the singularity is not used fundamentally. It
  is only used to get a constant larger and larger for
  $q \to + \infty$, because of the $q^s$ factor coming for $r_q^{-2s}$
  in the proof of Lemma~\ref{l:good-cone}.

  The second estimate given by Proposition~\ref{prop:G} is genuinely
  non-cutoff in nature. Indeed, it is adapted from \cite{luis} where
  the nonlinear maximum principle for singular integral operators in
  the spirit of \cite{constantinvicol} is used. In particular, the
  higher exponent on $|v|$ in Proposition~\ref{prop:G} is crucial in
  order to dominate the bad and non-singular terms for not-so-large
  $q$.
\end{remark}

\subsection{Lower bound on the good term $\mathcal{G}$ for $q=0$ and
  small $v$}

\begin{prop}[Estimate of {$\mathcal{G} (f,f)$} for $q=0$ and small
  $v$] \label{prop:Gsmallv} Assume $f$ satisfies \eqref{eq:contact}
  for $g = m$ (constant function $m=m(t_0)$ at the time of contact).
  Then we have
  \[ 
  \mathcal{G} (f,f)(v) \lesssim - m^{1+\frac{2s}d}.
  \] 
\end{prop}

\begin{proof}
  The proof is a variant of the previous one. We start from 
\begin{equation*}
  \mathcal{G}(f,f)(v) \le \PV  \int_{\R^d} \left[f(v')-m\right] K_{\bar f}(v,v') \dd
  v' \le 0.
\end{equation*}
where $m=g(t_0)$ is the upper bound barrier at the contact point.

We then use that
\[
  \left|\left\{ w \in \Xi(t,x,v) \, : \, f(v+w) \ge \frac{m}{2}
    \right\}\right| \le \frac2{m} \int_{w \in \Xi(t,x,v)} f(v+w) \dd w \le
  \frac{2M_0}{m}
\]
and using again equation~\eqref{eq:cone-estim} of Lemma
\ref{lem:cone-nd}, we can pick $r>0$ such that
\[
  \left|\Xi(t,x,v) \cap B_r \right| = \frac{8M_0}{m} \ \mbox{ with } \ r \approx
  \left(|v|^{-1} m^{-1}\right)^{\frac1d}
\]
and for this choice of $r$ we have
\[
  \left|\Xi(t,x,v) \cap B_r \right| \ge 4 \left|\left\{ w \in \Xi(t,x,v) \, :
      \, f(v+w) \ge \frac{m}{2} \right\}\right|.
\]
This implies that three fourth of the $w \in \Xi(v) \cap B_r$ satisfy
$ f(v+w) \le m/2$, and 
\begin{align*}
  \mathcal G(f,f)  & \leq - \frac{m}{2} \int_{\Xi(t,x,v) \cap B_r \cap \left\{f(v+w) \le \frac{m}{2}\right\}}
K_{\bar f}(v,v') \dd v' \\
  &\leq -\frac{m}{2} \int_{\Xi(t,x,v) \cap B_r \cap \left\{f(v+w) \le \frac{m}{2}\right\} } K_{\bar f}(v,v') \dd v' \\
  &\lesssim - m |v|^{\gamma+2s+1} r^{-d-2s} \left|\Xi(t,x,v) \cap B_r \cap
    \left\{f(v+w) \le \frac{m}{2}\right\}\right|
\end{align*} 
where we have used the estimate on the kernel $K_{\bar f}$ of
Lemma~\ref{lem:cone-nd}. We use now
\[
  \left|\Xi(t,x,v) \cap B_r \cap \left\{f(v+w) \le \frac{g(v)}{2} \right\}\right| \ge \frac34 |\Xi(t,x,v) \cap
  B_r| \approx r^d |v|^{-1}
\]
that follows from our choice of $r$ to deduce
\begin{align*}
  \mathcal G(f,f)(v) \lesssim - m |v|^{\gamma+2s} r^{-2s} \lesssim
  - m^{1+\frac{2s}{d}} 
\end{align*}
which concludes the proof.
\end{proof}

\subsection{Upper bound on the bad term $\mathcal B$ for large $q$}

We decompose further the bad term (see Figure~\ref{fig:bad})
\begin{equation}\label{eq:decomposition-bis}
  \mathcal{B}(f,f) = \mathcal{B}_1(f,f) + \mathcal{B}_2(f,f) +
  \mathcal{B}_3(f,f)
\end{equation}
with
\begin{align*}
    & \mathcal{B}_1(f_1,f_2)(v) \\ & :=  \PV \int_{v'_* \in \R^d} \tilde
                              \chi(v'_*) f_1(v'_*)
                      |v-v'_*|^{\gamma+2s} \int_{v' \in \; v +
                     (v-v'_*)^\bot} \chi_1(v')
                              \frac{\big[ f_2(v')-f_2(v) \big] \tilde b}{|v-v'|^{d-1+2s}}  \dd v' \dd v'_*. \\
   &  \mathcal{B}_2(f_1,f_2)(v) \\ & :=  \int_{v'\in \R^d} \chi_2(v') \frac{\big[ f_2(v')-f_2(v)
                         \big]}{|v'-v|^{d+2s}}\int_{v'_* \in \; v +
                     (v-v')^\bot} \tilde \chi(v'_*) f_1(v'_*)
                      |v-v'_*|^{\gamma+2s+1} \tilde b \dd v'_* \dd v'. \\
    & \mathcal{B}_3(f_1,f_2)(v) \\ & :=  \int_{v'\in \R^d}
                              \chi_3(v') \frac{\big[ f_2(v')-f_2(v)
                         \big]}{|v'-v|^{d+2s}}\int_{v'_* \in \; v +
                     (v-v')^\bot} \tilde \chi(v'_*)   f_1(v'_*)
                      |v-v'_*|^{\gamma+2s+1} \tilde b   \dd v'_* \dd v' 
\end{align*}
with $\tilde \chi(v'_*) := \one_{\{|v'_*|\ge c_1(q)|v|\}}$ (inherited
from good/bad decomposition) and the $v'$-integration domain is
decomposed as follows:
\begin{equation*}
  \chi_1(v') := \one_{\{|v'|>|v|/2\}},
  \quad \chi_2(v') := \one_{\{|v'| < c_3(q) |v|\}}, \quad \chi_3(v') =
  \one_{\{c_3(q) |v| \le |v'| <|v|/2\}}
\end{equation*}
with $c_3(q) := (1/2) (1+q)^{-1}$. 

It is intentional that the first term in the decomposition is written
with the $\int_{v'_*} \int_{v'}$ representation, while the second and
third is written with the $\int_{v'} \int_{v'_*}$ representation. This
corresponds to the respective representations used to estimate each
term below.

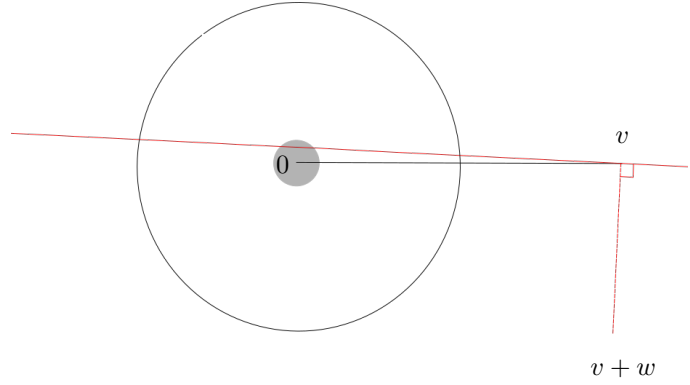
\begin{figure}
\setlength{\unitlength}{1.1cm} 
\begin{picture}(4.5,5)
\put(-2,1){\includegraphics[height=4.4cm]{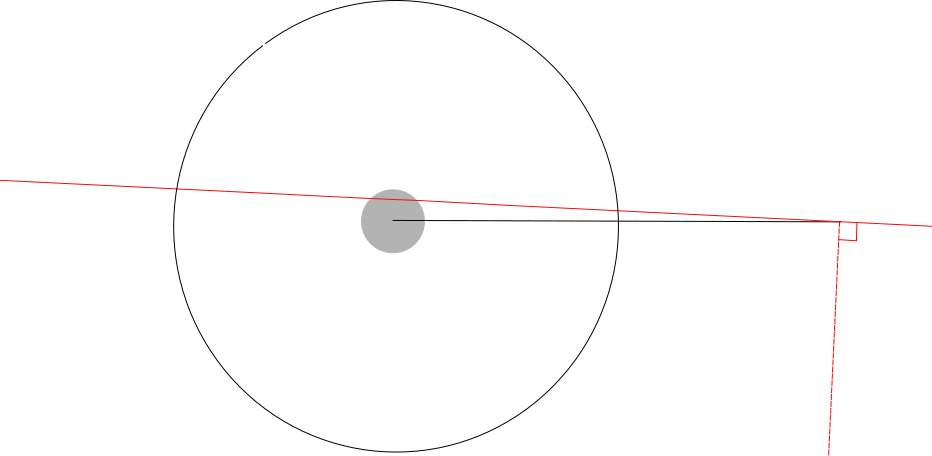}}
\put(1.2,2.94){$0$}
\put(5.3,3.3){$v$}
\put(5,0.5){$v+w$}
\end{picture}
\caption{The bad terms: $\mathcal{B}_1$ corresponds to the integration
  over the intersection of the line with the exterior of the balls,
  $\mathcal{B}_2$ corresponds to the integration over the intersection
  of the line with the grey ball, and $\mathcal{B}_3$ over the
  intersection of the line with the ring.}
\label{fig:bad}
\end{figure}

\smallskip

Observe that when $f \le g$ with contact at $v$, one has
$\mathcal{B}_1(f,f) \le \mathcal{B}_1(f,g)$. 
\begin{prop}[Estimate of {$\mathcal{B}_1(f,g)$} for all $q \ge
  0$] \label{prop:B1} Let $f$ be non-negative and satisfy \eqref{eq:non-deg}. Let
  $g = \A \min(1,|v|^{-q})$ with $q \ge 0$. For $|v| \ge 2$,
  \[ \mathcal{B}_1 (f,g)(v) \lesssim (1+q)^2 2^q |v|^{\gamma -2} g(v)\]
  with constant uniform in $q$. 
\end{prop}
%-------------------------------------------------
\begin{proof}
  Since $|v|\ge 2$ and (restriction of the domain $\chi_1$)
  $|v'| > |v|/2 >1$, we have $g(v) = N |v|^{-q}$ and
  $g(v') = N |v'|^{-q}$ and we further decompose the inner integral as
\begin{align*}
  I_1(v,v'_*) := &  \PV \int_{v' \in \; v + (v-v'_*)^\bot}  \chi_1
          (v') \frac{\big[
           g(v')-g(v) \big] \tilde b}{|v'-v|^{d-1+2s}} \dd v' \\    
   := &  \PV \int  \chi_1(v')
        \chi_{\{|v-v'| < |v|/2\}} \dots + \PV \int  \chi_1(v')
        \chi_{\{|v-v'| \ge  |v|/2\}} \dots \\ 
  =: & I_{1,<} + I_{1,>}.
\end{align*}

The term $I_{1,<}$ is estimated following the same argument as
in~\eqref{eq:taylor}. We subtract by symmetry the term
$\nabla g(v) \cdot (v'-v)$ that vanishes after integration
\begin{align*}
  I_{1,<}(v,v'_*) & =  \int_{v' \in \; v + (v-v'_*)^\bot} \chi_1
            (v') \chi_{\{|v-v'| < |v|/2\}}  \frac{\Big[  g(v')-g(v) - \nabla g (v) \cdot (v'-v) \Big]
                    \tilde b}{|v'-v|^{d-1+2s}} \dd v'
\end{align*}
and use~\eqref{eq:taylor} with $|v'_\theta| = |v+\theta (v'-v)| \ge |v|/2$:
\begin{align*}
  I_{1,<}(v,v'_*) & \lesssim 2^q \A \int_{v' \in \; v + (v-v'_*)^\bot}
                    \chi_1(v') \chi_{\{|v-v'| < |v|/2\}} |v|^{-q-2} |v'-v|^2 \frac{\dd v'}{|v'-v|^{d-1+2s}}\\
                  & \lesssim 2^q \A |v|^{-q-2s} \lesssim 2^q |v|^{-2s} g(v).
\end{align*}

The term $I_{1,>}$ is even simpler: the singularity is removed by the
restriction $|v'-v| \ge |v|/2 >1$ and the integrability at infinity is
provided by the kernel:
\begin{align*}
  I_{1,>}(v,v'_*) & \le \A \int_{v' \in \; v + (v-v'_*)^\bot} \chi_1(v') \chi_{\{|v-v'| \ge  |v|/2\}}
            |v'|^{-q}  \frac{\tilde b(\cos \theta)}{|v'-v|^{d-1+2s}}
            \dd v'\\
          & \lesssim 2^q \A |v|^{-q-2s} \le 2^q |v|^{-2s} g(v).
\end{align*}

We deduce that $I_1(v,v'_*) \lesssim 2^q |v|^{-2s}g(v)$ and we compute
\begin{align*}
  \mathcal{B}_1(f,g)(v) & =  \PV \int_{v'_* \in \R^d} \tilde \chi(v'_*) f(v'_*)  |v-v'_*|^{\gamma+2s} I_1(v,v'_*) \dd v'_* \\
                     & \lesssim 2^q g(v) |v|^{-2s} \int_{v'_* \in
                       \R^d} \tilde \chi(v'_*) f(v'_*)
                       |v-v'_*|^{\gamma+2s} \dd v'_* \\
                     & \lesssim c_1(q)^{-(\gamma+2s)} 2^q g(v) |v|^{-2s} \int_{v'_* \in
                       \R^d} \tilde \chi(v'_*)  f(v'_*)  |v'_*|^{\gamma+2s} \dd v'_* \\
                     & \lesssim c_1(q)^{-2} 2^q g(v) |v|^{\gamma-2} \int_{v'_* \in \R^d} f(v'_*)  (1+|v'_*|^2) \dd v'_* 
\end{align*}
where we have used in the last lines the fact that, under the
restriction $|v'_*| \ge c_1(q) |v|$ imposed by $\tilde \chi$, we have
\[
  |v-v'_*|^{\gamma+2s} \lesssim c_1(q)^{-(\gamma+2s)} (1+|v'_*|)^{\gamma+2s}
\lesssim c_1(q)^{-2} |v|^{\gamma+2s-2} (1+|v'_*|)^{2}.
\]
We deduce, since $c_1(q) = O(q^{-1})$ that 
\begin{equation*}
  \mathcal{B}_1(f,g)(v) \lesssim (1+q)^2 2^q g(v) |v|^{\gamma-2} (M_0 + E_0)
  \lesssim (1+q)^2 2^q g(v) |v|^{\gamma-2} 
\end{equation*}
which concludes the proof. 
\end{proof}

We recall that the next two estimates are based on the $\int_{v'}
\int_{v'_*}$ representation.
%--------------------------------------------------
\begin{prop}[Estimate of {$\mathcal{B}_2(f,f)$} for large
  $q$] \label{prop:B2} Let  $f$ be a non-negative function satisfying \eqref{eq:non-deg}. Let 
  $g = \A \min(1,|v|^{-q})$ with $q> d+ \gamma +2s$. Assume $f \leq g$ for all $v \in \R^d$. Then for
  $|v| \ge 2$,
  \[ \mathcal{B}_2 (f,f)(v)
    \lesssim  \frac{|v|^{\gamma}}{q-(d+\gamma + 2s)} g(v).\]
\end{prop}
%-------------------------------------------------
\begin{proof}
We first estimate from above the inner integral
\[ 
\begin{aligned}
I_2(v,v') &:= \int_{v'_* \in \; v + (v-v')^\bot} \tilde \chi(v'_*) f(v'_*)
|v-v'_*|^{\gamma+2s +1} \tilde b(\cos \theta) \dd v'_* , \\
&\leq \int_{v'_* \in \; v + (v-v')^\bot} \tilde \chi(v'_*) g(v'_*)
|v-v'_*|^{\gamma+2s +1} \tilde b(\cos \theta) \dd v'_* .
\end{aligned}
\]
We have for $|v'|< c_3(q) |v|$:
\begin{align*}
  |v'_*| &\ge |v'_*-v'|- c_3(q) |v| \\
         & \ge \big[|v-v'|^2+|v-v'_*|^2\big]^{1/2} - c_3(q) |v| \\
         & \ge \big[ (1-c_3(q))^2|v|^2 + |v-v'_*|^2\big]^{1/2} - c_3(q) |v|\\
  & \ge \left[ \frac12 |v|^2 + |v-v'_*|^2\right]^{1/2} - c_3(q) |v|\\
& \ge \left(1-\sqrt 2c_3(q)\right) \left[ \frac12|v|^2 + |v-v'_*|^2 \right]^{1/2}.
\end{align*}
Then we write ($|v'_*| >1$ for $v$ large enough from above, and thus
$g(v'_*) = \A |v'_*|^{-q}$):
\begin{align*}
  I_2(v,v') \lesssim & \int_{v'_* \in \; v + (v-v')^\bot} \tilde \chi(v'_*) g(v'_*)
        |v-v'_*|^{\gamma+2s +1} \tilde b(\cos \theta) \dd v'_* \\
  \lesssim & \A \int_{v'_* \in \; v + (v-v')^\bot} \tilde \chi(v'_*) |v'_*|^{-q}
             |v-v'_*|^{\gamma+2s +1} \dd v'_* \\
   \lesssim & \A \left(1-\sqrt{2} c_3(q) \right)^{-q} \int_{v'_* \in \; v + (v-v')^\bot} \tilde \chi(v'_*) 
        \left[ \frac12 |v|^2 + |v-v'_*|^2\right]^{-q/2}
              |v-v'_*|^{\gamma+2s +1} \dd v'_* \\
  \lesssim & \A \left(1-\sqrt{2} c_3(q) \right)^{-q} \int_{u \in
             \R^{d-1}}  
        \left[ \frac12 |v|^2 + |u|^2\right]^{-q/2} |u|^{\gamma+2s +1} \dd u \\
\lesssim & \A \left(1-\sqrt 2 c_3(q)\right)^{-q} |v|^{-q+\gamma+2s +d}
           \left( \int_{\hat u \in
             \R^{d-1}}  
        \left[ \frac12 + |\hat u|^2\right]^{-q/2} |\hat u|^{\gamma+2s +1}
           \dd \hat u \right)
\end{align*}
where the last integral is finite when $q -\gamma-2s -1 > d-1$, as
assumed, with
\begin{equation*}
  \left( \int_{\R^{d-1}} \left[ \frac12 + |\hat u|^2\right]^{-q/2} |\hat u|^{\gamma+2s +1}
           \dd \hat u \right) \lesssim \frac{1}{q - (\gamma+2s+d)}. 
\end{equation*}

Hence
\[ 
I_2(v,v') \lesssim \A \frac{\left(1-\sqrt 2 c_3(q)\right)^{-q}}{q-(d+\gamma + 2s)} |v|^{-q+\gamma+2s +d}. %=:J_2.
\]
We plug our estimate on $I_2$ into the formula for $\mathcal B_2$
(using the control of $|v-v'|^{-1} \le (1-c_3(q))^{-1} |v|^{-1}$ over the
restriction $\chi_2$):
\begin{align*} 
  \mathcal{B}_2 (f,f)(v)  &= \int_{\R^d} \chi_2(v')
                            \frac{\big[f(v')-f(v)\big]}{|v-v'|^{d+2s}} I_2(v,v') \dd v', \\
                          &\leq \A \left((1-c_3(q)\right)^{-d-2s} |v|^{-d-2s} \left( \int_{v'\in \R^d} f(v') \dd
                            v'\right) \sup_{v'} I_2(v,v') \\
                          & \lesssim  \frac{\A \left(1-c_3(q)\right)^{-d-2s} \left(1-\sqrt{2}c_3(q)\right)^{-q}}{q-(d+\gamma + 2s)}  |v|^{-q+\gamma}.
\end{align*}
The choice of
$c_3(q) = (1/2)(1+q)^{-1}$ shows that the factor
$\left(1-c_3(q)\right)^{-d-2s} \left(1-\sqrt{2}c_3(q)\right)^{-q}$ is
uniformly bounded for $q \ge 0$, which concludes the proof.
\end{proof}

%-------------------------------------------------
\begin{prop}[Estimate of $\mathcal{B}_3(f,f)$ for large $q$] \label{prop:B3} Let $f$ be a non-negative function satisfying \eqref{eq:non-deg}. Assume $f \leq g$ for all $v \in \R^d$, where {$g = \A \min(1,|v|^{-q})$} for
  $q> d+\gamma +2s$. Then for all $|v| \ge 2$,
\[ \mathcal{B}_3 (f,f)(v) \lesssim (1+q)^2 \left( (1+q)^{q-(d-1)} +
      \frac{1}{q-(d+\gamma+2s)} \right) |v|^{\gamma -2} g(v).\]
\end{prop}
%-------------------------------------------------
\begin{proof}
  We first estimate
  \[
\begin{aligned}
    I_3(v,v') &:= \int_{v'_* \in \; v + (v-v')^\bot} \tilde \chi(v'_*)
    f(v'_*) |v-v'_*|^{\gamma+2s +1} \tilde b(\cos \theta) \dd v'_*, \\
&\leq \int_{v'_* \in \; v + (v-v')^\bot} \tilde \chi(v'_*)
    g(v'_*) |v-v'_*|^{\gamma+2s +1} \tilde b(\cos \theta) \dd v'_* 
\end{aligned}
  \]
  under the conditions $c_3(q) |v| < |v'| < |v|/2$ and
  $|v'_*| \ge c_1(q) |v|$ imposed by $\chi_3$ and $\tilde \chi$. We
  change variable $v'_* = v + |v|\tilde u$ and bound from above
  (denoting $\hat v:=v/|v|$)
  \[
    I_3(v,v') \lesssim \A |v|^{\gamma+2s+d -q} \int_{\tilde u \; \in \;
      (v-v')^\bot} \tilde \chi(v'_*) \left| \hat v + \tilde u
    \right|^{-q} \left| \tilde u \right|^{\gamma+2s +1} \dd \tilde u.
  \]
  The restriction $\tilde \chi(v'_*)$ imposes
  $|\hat v + \tilde u| \ge c_1(q) >0$. Close to the singularity
  $|\hat v + \tilde u| \sim c_1(q)$, then $|\tilde u| \sim 1$ and the
  integral in $\tilde u$ is controlled by $O(c_1(q)^{d-1-q})$. At large
  $\tilde u$, the integral is finite provided that $q >
  d+\gamma+2s$:
  \[
    \int_{\tilde u \in \;
      (v-v')^\bot} \tilde \chi(v'_*) \left| \hat v + \tilde u
    \right|^{-q} \left| \tilde u \right|^{\gamma+2s +1} \dd \tilde u
    \lesssim c_1(q)^{d-1-q} + \frac{1}{q-(d+\gamma+2s)}.
  \]

  We finally plug this estimate into the formula for $\mathcal{B}_3$:
  \begin{align*}
   & \mathcal{B}_3(f,f)(v) \\ &\le \int_{v'\in \R^d} \chi_3(v') \frac{\big[
      f(v')-f(v) \big]}{|v'-v|^{d+2s}} I_3(v,v') \dd v'\\
  &\lesssim \A c_3(q)^{-2} \left( c_q^{d-1-q} + \frac{1}{q-(d+\gamma+2s)}
    \right) |v|^{\gamma-2-q} \left( \int_{v'\in \R^d} f(v')(1+|v'|^2) \dd
    v'\right) 
  \end{align*}
  where we have used the restriction $|v'|<|v|/2$ on $\chi_3(v')$ to
  deduce $|v-v'| \sim |v|$, and the restriction $|v'| >c_3(q)|v|$ to
  deduce $1 \le c_3(q)^{-2} |v|^{-2}|v'|^2$. From the assumption on the
  mass and energy of $f$ we get finally
  \[
    \mathcal{B}_3(f,f)(v) \lesssim c_3(q)^{-2} \left( c_1(q)^{d-1-q} +
      \frac{1}{q-(d+\gamma+2s)} \right) |v|^{\gamma-2} g(v)
  \]
which concludes the proof.     
\end{proof}

\subsection{Upper bound on the bad term $\mathcal{B}$ for not-so-large
  $q$}

%----------------------------------------------------------------------------
\begin{prop}[Estimate of {$\mathcal{B}_2(f,g) + \mathcal{B}_3(f,g)$} for  not-so-large $q$]\label{prop:B23}
Let $f$ be a non-negative function satisfying \eqref{eq:non-deg}. Let  $g = \A \min(1,|v|^{-q})$ and $q \in [0,d+1]$. Then for $|v| \ge 2$,
  \[
    (\mathcal{B}_2 + \mathcal{B}_3)(f,g)(v) \lesssim
    \begin{cases}
      |v|^{\gamma-(d+1-q)}g(v) & \text{if } q > d-1,\\
      |v|^{\gamma-2} \ln (1+|v|)g(v) & \text{if } q = d-1,\\
      |v|^{\gamma-2}g(v) & \text{if } q < d-1.
    \end{cases}
  \]
\end{prop}
%----------------------------------------------------------------------------
\begin{proof}
  Denote $\mathcal{B}_{2+3} := \mathcal{B}_2 + \mathcal{B}_3$ and
  $\chi_{2+3}(v') :=\chi_2(v') + \chi_3(v') = \one_{\{|v'|
    <|v|/2\}}$. Then
  \begin{align*}
    & \mathcal{B}_{2+3}(f,g)(v) \\
    & = \int_{|v'_*| > c_1(q)|v|} f(v'_*)
    |v-v'_*|^{\gamma+2s} \left\{ \int_{v' \in \; v + (v-v'_*)^\bot}
    \chi_{2+3}(v') \frac{\left[g(v')-g(v)\right] \tilde
      b}{|v'-v|^{d-1+2s}} \dd v' \right\} \dd v'_*.
  \end{align*}
  We use that
  $g(v') - g(v) \leq 2^q \A (1+|v'|)^{-q}$ in order
  to write
  \begin{align*}
    & \mathcal{B}_{2+3}(f,g)(v)\\
    & \lesssim  2^q \int_{|v'_*| > c_1(q) |v|} f(v'_*) |v-v'_*|^{\gamma+2s} 
      \left\{ \int_{v' \in \; v + (v-v'_*)^\bot} \chi_{2+3}(v')
      \frac{\A (1+|v'|)^{-q}}{|v'-v|^{d-1+2s}} \right\} \dd v'  \dd v'_*, \\
    &\lesssim 2^q \A |v|^{-d+1-2s} \int_{|v'_*| > c_1(q) |v|} f(v'_*)
      |v-v'_*|^{\gamma+2s}
      \left\{ \int_{v' \in \; v+(v-v'_*)^\bot} \chi_{2+3}(v') (1+|v'|)^{-q} \dd v' \right\} \dd v'_*.
\end{align*}

We get 
\[
  \int_{v' \in \; v + (v-v'_*)^\bot} \chi_{2+3}(v') (1+|v'|)^{-q} \dd
  v' \lesssim \Theta(v)
\ \mbox{ with } \ 
  \Theta(v) :=
  \begin{cases}
    1 & \text{ if } q > d-1,\\
    \ln (1+|v|) & \text{ if } q = d-1, \\
    |v|^{d-1-q} &\text{ if } q < d-1.
  \end{cases}
\]

We deduce that 
\begin{align*}
  \mathcal{B}_{2+3}(f,g) (v)
  &\lesssim \A |v|^{-d+1-2s} \Theta(v)
    \int_{|v'_*| > c_1(q) |v|} f(v'_*) |v-v'_*|^{\gamma+2s} \dd v'_*, \\
  &\leq \A |v|^{-d+1-2s}  \Theta(v) \left( \max_{|v'_*| > c_1(q) |v|}
    \frac{|v-v'_*|^\gamma}{|v'_*|^2} \right)
    \int_{|v'_*| > c_1(q) |v|} f(v'_*) \left(1+|v'_*|^2\right) \dd v'_*, \\
  &\lesssim \A |v|^{-d-1+\gamma} \Theta(v)
\end{align*}
which concludes the proof.
\end{proof}

\subsection{Upper bound on the non-singular (lower order) term
  $Q_{ns}$}

%-------------------------------------------------------------------------------------------------
\begin{prop}[Estimate of $Q_{ns}(f,f)$] \label{prop:Q2} Assume $f$
  satisfies \eqref{eq:contact} with $g = \A \min(1,|v|^{-q})$. Then
  for $\gamma \ge 0$ 
\[ Q_{ns}(f,f)(v) \lesssim  (1+|v|)^{\gamma} g(v), \]
while for $\gamma <0$,
\[
  Q_{ns}(f,f) \lesssim 2^{-\frac{q\gamma}{d}} g(v)^{1-\frac{\gamma}{d}} +
  (1+|v|)^\gamma g(v).
\]

Moreover when $q=0$ and $\gamma\in (-d,0)$, by using the uniform bound
on the local entropy it is possible to weaken slightly the dependency
on $g(v)$ as follows: there is a function $\psi =\psi(r)$ on $\R_+^*$
that goes to zero as $r \to +\infty$ such that
\begin{equation}
  \label{eq:Qnslog}
  Q_{ns}(f,f) \lesssim g(v)^{1-\frac{\gamma}{d}} \psi(g(v))
  +  (1+|v|)^\gamma g(v).
\end{equation}
The function $\psi$ is explicit from the proof and depends on $M_0$
and $H_0$. 
\end{prop}
%--------------------------------------------------------------------------------------------------
\begin{proof}
We first deal with the easier case $\gamma \ge 0$:
\begin{align*}
  Q_{ns}(f,f)(v) & = C_S f(v) \int_{\R^d} f(v-v_*) |v_*|^\gamma \dd v_* \\
                 & \lesssim g(v) \int_{\R^d} f(v-v_*) \left(|v-v_*|^\gamma+ |v_*|^{\gamma}\right) \dd v_* \\
                 & \lesssim g(v) \int_{\R^d} f(v-v_*) \left(|v-v_*|^2+1 + |v|^{\gamma}\right) \dd v_*, \\
                 & \lesssim g(v) \left[ E_0 + \left(1+|v|^\gamma\right) M_0 \right] \\
                 & \lesssim |v|^\gamma g(v) 
\end{align*}
where we have used $|v| \ge 1$.

We now turn to the case $\gamma <0$ and pick $r < |v|/2$ and write
\begin{align*}
  &  Q_{ns}(f,f)(v) \\
  & = C_S f(v) \int_{|v-v_*|< r} f(v_*) |v-v_*|^\gamma
    \dd v_* + C_S f(v)  \int_{|v-v_*| > r} f(v_*) |v-v_*|^\gamma \dd v_*   \\
  & \lesssim  2^ q g(v)^2 \int_{|v-v_*|< r} |v-v_*|^\gamma \dd v_*  
    + g(v) r^{\gamma} \int_{\R^n} f(v_*)  \dd v_* \\
  & \lesssim 2^q g(v)^2 r^{d+\gamma} + M_0 r^{\gamma}
\end{align*}
where we have used in the first integral the fact that $|v_*| \ge
|v|/2$, and $|v| \ge 1$. The optimisation in $r$ gives, for 
\[
r := \min \left[ \left(\frac{M_0}{2^q g(v)}\right)^{1/d} \; , \;
  \frac{|v|}{2} \right],
\]
the estimate
\begin{equation*}
  Q_{ns}(f,f)(v) \lesssim 2^{-q\gamma/d} g(v)^{1-\gamma/d}
\end{equation*}
when $g(v) \ge M_0 2^{-q} (|v|/2)^{-d}$, and otherwise it gives 
\begin{equation*}
  Q_{ns}(f,f)(v) \lesssim g(v) (1+|v|)^\gamma
\end{equation*}
which concludes the proof of the second inequality. 

Let us finally consider the proof of the third refined inequality in
the case $\gamma<0$. We use again the classical fact that the entropy
bound implies the non-concentration estimate
\[
  \int_A f(t,x,v) \dd v \lesssim_{M_0,H_0} \varphi(|A|) \quad \mbox{
    with } \varphi(r) = \ln(1+r) + \left[ \ln\left(r^{-1}\right) \right]^{-1}
\]
and $A$ a Borel set. Split the integral as
\begin{equation*}
    Q_{ns}(f,f)(v)
   = C_S f(v) \int_{|v-v_*|< r_1} \dots  + C_S f(v) \int_{r_1
    \le |v-v_*|< r_2} \dots  + C_S f(v)  \int_{|v-v_*| > r_2} \dots 
\end{equation*}
with (for $g(v)$ large enough, otherwise the previous estimate is
sufficient):
\[
  r_1 := \left( \frac{M_0}{g(v)} \right)^{\frac{1}{d}}
  \left[\varphi\left(\frac{M_0}{g(v)}\right)\right]^{-\frac{1}{2\gamma}}
  < r_2 := \left( \frac{M_0}{g(v)} \right)^{\frac{1}{d}}
  \left[\varphi\left(\frac{M_0}{g(v)}\right)\right]^{\frac{1}{2\gamma}}
\]
and apply the $L^\infty$ bound in the first term, the
non-concentration estimate in the second term and the $L^1$ bound on
the third term to get
  \[
    Q_{ns}(f,f) \lesssim g(v)^{1-\frac{\gamma}{d}}
  \left[\varphi\left(\frac{M_0}{g(v)}\right)\right]^{\min\left(\frac{1}{2},\frac{(d+\gamma)}{|\gamma|}\right)}
  +  (1+|v|)^\gamma g(v).
\]
This concludes the proof of this third inequality.
\end{proof}

\section{Maximum principle and proof of the upper bounds}
\label{s:proof}

\subsection{The strategy}

We recall that the strategy is to prove that the solution $f$ remains
below a certain barrier function $g$ ensuring the upper bound
$N(t) (1+|v|)^{-|q|}$ for $q \ge 0$ and $N(t)$ a function of time that
is either constant (for propagation of pointwise moments) or singular
 (for the appearance of pointwise moments) at $t \to 0$.

We consider a first contact point $(t_0,x_0,v_0)$ such
that~\eqref{eq:contact} holds true. Recall that the existence of this first contact point is guaranteed by the rapid decay assumption in Definition \ref{d:solutions-rapid-decay} and the compactness of the spatial domain. At this point, the inequality \eqref{e:to-contradict} would hold. We use the fine structure of the collision operator $Q(f,f)$ to obtain that it is ``negative enough'' at large velocities. Concretely, we prove that the
negative ``good part'' $\mathcal{G}$ dominates the other ``bad'' and non-singular parts of the collision operator at large velocities. Note that, for higher pointwise moments, the not-so-large velocities are controlled thanks to the $L^\infty$ bound in Theorem~\ref{thm:linfty}.

 We start by revisiting the $L^\infty$ bound
of~\cite{luis} in order to include the minor technical extensions
needed for this paper.

\subsection{The $L^\infty$ bound from \cite{luis}}

The first proof of the $L^\infty$ bound for solutions
satisfying~\eqref{eq:non-deg} was obtained by the third author
in~\cite[Theorem~1.2]{luis}. We state here a slightly refined version. 
%----------------------------------------------------------
\begin{thm}[$L^\infty$ bound] \label{thm:linfty} Let
  $\gamma \in \R$ and $s \in (0,1)$ satisfy $\gamma +2s \in [0,2]$ and
  $f$ be a non-negative solution of the Boltzmann
  equation~\eqref{eq:boltzmann} such that \eqref{eq:non-deg} holds
  true for some positive constants $m_0,M_0,E_0,H_0$. Then 
  \[ \forall \, t \in (0,T], \qquad \|f(t,\cdot)\|_{L^\infty} \le \Ai \left(1+
      t^{-\frac{d}{2s}}\right) \] for positive constant $\Ai$ only
  depending on $m_0,M_0,E_0,H_0$, dimension, $\gamma$ and $s$.

Moreover, if $\|f_0\|_{L^\infty} < \A$ for $\A \ge \A_\infty$, then $\|f(t,\cdot)\|_{L^\infty} < \A$ for $t\in[0,T]$.
\end{thm}
%----------------------------------------------------------
\begin{remark}
  With respect to~\cite[Theorem~1.2]{luis}: the marginal improvements
  are the inclusion of the borderline case $\gamma+2s=0$ and the fact
  that if the initial data is bounded, the $L^\infty$ bound is uniform
  as $t \downarrow 0^+$. We provide a detailed proof below for
  self-containedness and because of these small variations. 
\end{remark}

\begin{proof}
Without loss of generality, it is enough to show the inequality holds
for $t \in (0,1]$. We consider the barrier
$g(t,v) := \A_\infty t^{-\frac d {2s}}$ and consider the equation
~\eqref{e:to-contradict} at a first contact point $t_0 \in (0,1]$ and
$v_0 \in \R^d$. It is enough to prove that
for $\A_\infty$ large enough
\begin{equation}\label{eq:contradiction}
  Q(f,f)(t_0,x_0,v_0) < -\frac d {2s} \A_\infty t_0^{-\frac d {2s}-1}.
\end{equation}

Observe that when $g$ is constant in $v$, at the contact point the bad
term satisfies
\[
  \mathcal{B}(f,f)(t_0,x_0,v_0) \le 0
\]
and can be discarded. We then apply Proposition~\ref{prop:G} for
$|v_0| \ge 1$ and Proposition~\ref{prop:Gsmallv} for $|v_0| \le 1$ to
get
\begin{equation*}
  \mathcal G(f,f)(t_0,x_0,v_0) \lesssim - \A_\infty^{1+\frac{2s}{d}} t_0^{-\frac d {2s}-1}
  (1+|v_0|)^{\gamma+2s +\frac{2s}{d}}
\end{equation*}
and Proposition~\ref{prop:Q2} (valid for all $v$) to get 
\begin{equation}\label{eq:LinftyQns1}
  Q_{ns}(f,f)(t_0,x_0,v_0) \lesssim \A_\infty t_0^{-\frac d {2s}}(1+|v_0|)^{\gamma} +
  \one_{\gamma <0} \A_\infty ^{1-\frac{\gamma}{d}} t_0^{-\frac d {2s}+\frac{\gamma}{2s}}.
\end{equation}

In case $\gamma + 2s > 0$, the exponents of $\A_\infty$ and $t_0^{-1}$
in the first negative equation are strictly greater than those 
in second positive equation. Moreover the exponent
$\gamma+2s+2s/d$ of $|v_0|$ is strictly greater than all the other
exponents $\gamma-2$, $\gamma$ and $0$. Thefore by choosing
$\A_\infty$ large enough, we deduce that
\[
  Q(f,f)(t_0,x_0,v_0) \le - \frac12 \A_\infty^{1+\frac{2s}{d}}
  t_0^{-\frac d {2s}-1} (1+|v_0|)^{\gamma+2s +\frac{2s}{d}}
\]
and taking $\A_\infty$ even greater if necessary, this
contradicts~\eqref{eq:contradiction}. 

The case $\gamma+2s=0$ (and thus $\gamma <0$) is treated similarly but
since the inequality~\eqref{eq:LinftyQns1} is now too weak to show
that $Q_{ns}(f,f)$ is dominated by $\mathcal{G}(f,f)$ for large
$\A_\infty$, we use instead the refined inequality~\eqref{eq:Qnslog}
from Proposition~\ref{prop:Q2} to get
\[
  Q_{ns}(f,f)(t_0,x_0,v_0) \lesssim \left( \A_\infty t^{-\frac d {2s}}
  \right)^{1+\frac{\gamma}d} \psi \left( \A_\infty t^{-\frac d
      {2s}} \right).
\]
With this inequality, we recover that $Q_{ns}(f,f)$ is dominated by
$\mathcal{G}(f,f)$ for $\A_\infty$ sufficiently large and the
contradiction follows as before.

We finally prove the propagation of the $L^\infty$ bound when it is finite
initially. If $\|f_0\|_{L^\infty} < \A$ for some $\A \geq \A_\infty$,
we pick $t_0 \in (0,1)$ such that
$\A_\infty t_0^{-\frac d {2s}} = \A$. By the same reasoning as before,
we obtain
\[ f(t,x,v) < \A_\infty (t+t_0)^{-\frac d {2s}}.\] In particular,
$f(t,x,v) < \A$ for $t \in (0,1-t_0)$. This allows us to extend the
upper bound for a fixed period of time. Iterating this, we extend it
for all time.
\end{proof}

\begin{remark} Here we present some further interpretation of the cone
  of non-degeneracy and the $L^\infty$ bound. The cone of
  Lemma~\ref{lem:D} is a cone of direction for $(v'-v)$, i.e. the
  so-called ``$\omega$'' vector of the ``$\omega$-representation''
  (see \cite[Section 4.6]{villani-book}):
\begin{align*}
  A(v) :=
  \left\{ \omega \in \mathbb S^{d-1} \
  \mbox{s.t.} \ \left| \left\{ v'_* \ : \ (v'_*-v_*)
  \cdot \omega =0 \ \& \ f(v'_*) > c_0 \ \& \ |v'_*| < R_0 \right\}
  \right| > \delta \right\}. 
\end{align*}
(The variable $v'$ remains to be integrated independently of this
cone.) This is a set of directions where the kernel is bounded below
in the Carleman representation. The fact that the set where $f$ is
bounded below can be some complicated Borel set in a ball near zero
does not change fundamentally the argument, which would be very
similar if $f \ge \ell \chi_{B_r}$. The set $A(v)$ is
$\{ \omega \in \mathbb S^{d-1}: |\omega \cdot \frac{v}{|v|}| \lesssim |v|^{-1}\}$ or a non-zero
measure-proportion of this set, hence $|\omega \cdot v| \lesssim R_0$
or a non-zero proportion of this set of directions.

The goal of this cone of direction is to find configurations so
that $v'_*$ is brought back near $0$ in a zone where $f$ is bounded
below, in order to bound from below the ``coefficients'' of the
operator, i.e. the kernel. 

Then this set of directions $A(v)$ creates a cone $v' \in \Xi(v)$
centred at $v$ and of angles of order $1/|v|$ close to orthogonal to
$v/|v|$. Then in \cite{luis}, see Theorem~\ref{thm:linfty} above,
the part of this cone where $f < (1/2) \max f$ is bounded below using
the Chebychev inequality and the mass and energy bounds. That is: the
assumptions imply that $f$ is, for a significant amount of the large
velocities, far from its maximum, i.e. less than $(\max f)/2$. On this
part of the cone, the coercivity of $Q_1(f,f)$ is recovered, and together
with the bounds from above on $Q_{ns}(f,f)$, gives the contradiction and the
$L^\infty$ barrier.
\end{remark}

\subsection{Appearance of pointwise moments when $\gamma>0$}
\label{s:appearance-hard}

Let us now prove the appearance of pointwise moments (second part of
the theorem), when assuming furthermore that $\gamma >0$ and
restricting without loss of generality to $t \in [0,1]$. Consider
$g(v) = \A(t) \min\left(1,|v|^{-q} \right)$ where
$\A(t) = \Ao t^{-\beta}$ and $\beta = \frac q \gamma + \frac{d}{2s}$,
and with $q >d+1$ to be chosen large enough later. We recall that the
existence of the first contact point is granted by our assumptions on
the solution (periodic condition in $x$ and rapid qualitative decay in
$v$).

At the first contact point $g(t_0,v_0)=f(t_0,x_0,v_0)$ and Theorem
\ref{thm:linfty} implies that
$\A(t) \min(1,|v|^{-q}) \leq \Ai t^{-\frac{d}{2s}}$ which shows that
$|v_0| \gtrsim N_0^{1/q} t_0^{-\gamma}$ can be made large by choosing
$N_0$ large enough. In particular we can apply again
Propositions~\ref{prop:G-bis}, \ref{prop:B1}, \ref{prop:B2},
\ref{prop:B3} and~\ref{prop:Q2} to get
\begin{align*}
\mathcal G(f,f)(t_0,x_0,v_0) &\lesssim - q^s |v_0|^{\gamma} g(v_0) && \text{from Proposition } \ref{prop:G-bis},\\
\mathcal B_1 (f,f)(t_0,x_0,v_0) &\lesssim q^2 2^q |v_0|^{\gamma-2} g(v_0) && \text{from Proposition } \ref{prop:B1}, \\
\mathcal B_2 (f,f)(t_0,x_0,v_0) &\lesssim \frac 1 q |v_0|^\gamma g(v_0) && \text{from Proposition } \ref{prop:B2},\\
\mathcal B_3 (f,f)(t_0,x_0,v_0) &\lesssim q |v_0|^{\gamma-2} g(v_0) && \text{from Proposition } \ref{prop:B3}, \\
Q_{ns}(f,f)(t_0,x_0,v_0) &\lesssim |v_0|^{\gamma} g(v_0) && \text{from Proposition } \ref{prop:Q2},
\end{align*}
and therefore by choosing $q$ large enough (independently of $\Ao$) we
deduce
\begin{align*}
  Q(f,f)(t_0,x_0,v_0)  \lesssim  \left[ - q^s |v_0|^{\gamma} +
  |v_0|^\gamma + |v_0|^{\gamma-2} \right] g(v_0) \lesssim  - q^s
  |v_0|^{\gamma} g(v_0)
\end{align*}
which yields the inequality
\[
  -\left(\frac{q}{\gamma} + \frac{d}{2s} \right) \frac{g(v_0)}{t_0} =
  \partial_t g(t_0,v_0) \le
  Q(f,f)(t_0,x_0,v_0) \le - C q^s |v_0|^{\gamma} g(v_0)
\]
for some constant $C>0$ at the contact point. Since $|v_0|  \gtrsim
N_0^{1/q} t_0^{-\gamma}$ we deduce that
\[
  \left(\frac{q}{\gamma} + \frac{d}{2s} \right) \ge C' q^s
  N_0^{\gamma/q}
\]
which is a contradiction for $N_0$ large enough. This shows that the
contact point does not exist and concludes the proof of the appearance
of pointwise moments.

\begin{remark}
  Note that this proof only uses the first estimate on the good term
  (Proposition~\ref{prop:G-bis}), and therefore does not fully exploit
  the non-cutoff nature of the collision operator.
\end{remark}

\subsection{Propagation of pointwise moments}
\label{s:propag}

We consider the setting and assumptions of Theorem~\ref{thm:upper} and
prove first the propagation of pointwise moments (first part of the
theorem). Consider $g(v) = \Ao \min\left(1,|v|^{-q} \right)$ with
$q >d+1$ to be chosen large enough later. At the first contact point
$g(t_0,v_0)=f(t_0,x_0,v_0)$ and Theorem~\ref{thm:linfty} implies that
$\Ao \min(1,|v_0|^{-q}) \leq \Ai$ which shows that
$|v_0| \gtrsim N_0^{1/q}$ can be made large by choosing $N_0$ large
enough. Apply Propositions~\ref{prop:G-bis}, \ref{prop:B1},
\ref{prop:B2}, \ref{prop:B3} and~\ref{prop:Q2} apply at this contact
point:
\begin{align*}
\mathcal G(f,f)(t_0,x_0,v_0) &\lesssim - q^s |v_0|^{\gamma} g(v_0) && \text{from Proposition } \ref{prop:G-bis},\\
\mathcal B_1 (f,f)(t_0,x_0,v_0) &\lesssim q^2 2^q |v_0|^{\gamma-2} g(v_0) && \text{from Proposition } \ref{prop:B1}, \\
\mathcal B_2 (f,f)(t_0,x_0,v_0) &\lesssim \frac 1 q |v_0|^\gamma g(v_0) && \text{from Proposition } \ref{prop:B2},\\
\mathcal B_3 (f,f)(t_0,x_0,v_0) &\lesssim q |v_0|^{\gamma-2} g(v_0) && \text{from Proposition } \ref{prop:B3}, \\
Q_{ns}(f,f)(t_0,x_0,v_0) &\lesssim |v_0|^{\gamma} g(v_0) + 
                           \one_{\gamma<0} 2^{-\frac{q\gamma}{d}}
                           g(v_0)^{1-\frac{\gamma}{d}} && \text{from Proposition } \ref{prop:Q2}.
\end{align*}

We choose $q$ large enough (independently of $N_0$) so that $\mathcal G(f,f) + \mathcal B_2(f,f) + |v_0|^\gamma g(v_0) \lesssim - q^s |v_0|^\gamma g(v_0)$. For large $|v_0|$ (ensured by our choice of $N_0$, that depends on $q$), we get
\begin{align*}
  Q(f,f)(t_0,x_0,v_0)  = & \quad \mathcal G(f,f)(t_0,x_0,v_0) + \mathcal B_1 (f,f)(t_0,x_0,v_0)
                    + \\ 
  & \quad \mathcal B_2 (f,f)(t_0,x_0,v_0) + \mathcal B_3 (f,f)(t_0,x_0,v_0) +
  Q_{ns}(f,f)(t_0,x_0,v_0) \\  \lesssim & - q^s |v_0|^{\gamma} g(v_0) <0
\end{align*}
which contradicts the inequality
$0= \partial_t g(t_0,v_0) \le \partial_t f(t_0,x_0,v_0) =
Q(f,f)(t_0,x_0,v_0)$ at this contact point. This shows that the
contact point does not exist and concludes the proof of the
propagation of pointwise moments.

When $\gamma <0$ and $q \ge d+1$ large enough the only additional
difficulty is the second term on the right hand side of the control on
$Q_{ns}$. But
\[
  g(v_0)^{1-\frac{\gamma}{d}} \lesssim \Ao^{1-\frac{\gamma}{d}} |v_0|^{-q + \frac{q}{d}\gamma}
\]
and the exponent of $|v_0|$ is strictly lower than that of
$\mathcal{G}(f,f)$, uniformly in $q \ge d+1$, so is dominated by
$\mathcal{G}(f,f)$ by taking $|v_0|$ large enough (through $\Ao$ large
enough). Finally taking $q$ large enough yields the same contradiction
as before.

\subsection{Appearance of low pointwise moments for $\gamma \le 0$}

Consider as before $g(v) =\A (t) \min(1,|v|^{-q})$ with $q \ge 0$ to
be restricted later, and $\A(t) = \Ao t^{-\frac d {2s}}$ and
$\Ao = \Ao(m_0,M_0,E_0,H_0,\gamma,s,d)$ is a large constant to be
determined below. As before it is sufficient to prove that the
conclusion holds for $t \in (0,1]$.

At the first contact point $g(t_0,v_0)=f(t_0,x_0,v_0)$ and Theorem
\ref{thm:linfty} implies that $\Ao \min(1,|v_0|^{-q}) \leq \Ai$ which
shows that $|v_0| \gtrsim N_0^{1/q}$ can be made large by choosing
$N_0$ large enough. Apply Propositions~\ref{prop:G}, \ref{prop:B1},
\ref{prop:B2}, \ref{prop:B23} and~\ref{prop:Q2} at this contact point
(note that we do not track the dependency in $q$ since it is bounded here):
\begin{align*}
  \mathcal G(f,f)(t_0,x_0,v_0)
  &\lesssim - |v_0|^{(\gamma+2s)+\frac{2s}d} g(v_0)^{1+\frac{2s}d}
  && \text{from Proposition \ref{prop:G}}, \\
  \mathcal B_1 (f,f)(t_0,x_0,v_0)
  &\lesssim |v_0|^{\gamma-2} g(v_0)
  && \text{from Proposition \ref{prop:B1}}, \\
  (\mathcal B_2 + \mathcal B_3 )(f,f)(t_0,x_0,v_0)
  &\lesssim |v_0|^{\gamma} g(v_0)
  && \text{from Proposition \ref{prop:B23},}\\
  Q_{ns}(f,f)(t_0,x_0,v_0)
  &\lesssim |v_0|^{\gamma} g(v_0) + g(v_0)^{1-\frac{\gamma}d}
  && \text{from Proposition \ref{prop:Q2}}.
\end{align*}

To check that the first negative term dominates the other terms
(i.e. is larger than, say, twice all the other terms for $\Ao$ large
enough), there are three \emph{independent} conditions to check: (1)
that the (negative) exponent of $|v_0|$ is strictly greater in this
negative term than the corresponding exponents in all the other terms,
and (2) that the (positive) power of $\A(t)$ is strictly greater in
this negative term than the corresponding exponents in all the other
terms, and finally (3) that the exponent of $|v_0|$ in the negative
term is greater or equal than that of the barrier, i.e. $q$. Note in
particular that the two first conditions must be checked independently since
$|v_0|$ can be possibly be much larger than $\Ao^{1/q}$. As far as (1)
is concerned, check that
\begin{align*}
  \gamma + 2s + \frac{2s}{d} -q -q\frac{2s}{d} 
  & > \gamma - 2 - q 
  && \text{for all $q \in [0,3(d+1)) \supset [0,d+1)$}, \\
    \gamma + 2s + \frac{2s}{d} -q -q\frac{2s}{d} 
  & > \gamma - q 
  && \text{for all $q \in [0,d+1)$}, \\
    \gamma + 2s + \frac{2s}{d} -q -q\frac{2s}{d} 
  & > - q + q \frac{\gamma}{d}
  && \text{for all $q \ge \left[0,d+\frac{2s}{\gamma+2s}\right)
     \supset [0,d+1)$}.
\end{align*}
As far as (2) is concerned, check that
\begin{align*}
  1+\frac{2s}{d}  
  & > 1 
  && \text{for all $q \in \R_+ \supset [0,d+1)$}, \\
  1 + \frac{2s}{d} 
  & \ge 1- \frac{\gamma}{d}
  && \text{for all $q \in \R_+ \supset [0,d+1)$},
\end{align*}
where we have used $\gamma + 2s \ge 0$ in the last inequality. As far
as (3) is concerned, check that
\begin{align*}
  \gamma + 2s + \frac{2s}{d} - q - q \frac{2s}{d}   
  & \ge -q
  && \text{for all $q \in \left[0,d+1 + \frac{d\gamma}{2s}\right]$.}
\end{align*}

We thus impose the most restrictive condition
$q = d+1 + \frac{d\gamma}{2s}$ if $\gamma <0$. In the limit case
$\gamma=0$, observe however that if the condition (1) above is
saturated (same exponents of $|v_0|$) and the condition (3) is
satisfied, but the condition (2) is strict (strictly greater exponent
of $\A(t)$ is the negative term), we can still prove that the negative
term dominates by taking $\Ao$ large enough. This proves in all cases
that, by choosing $\Ao$ and thus $|v_0|$ large enough:
\begin{align*}
Q(f,f)(t_0,x_0,v_0) &\lesssim -|v_0|^{(\gamma+2s)+\frac{2s}d} g(v)^{1+\frac{2s}d}, \\
& \lesssim -\A(t)^{1+\frac{2s}d} |v_0|^{-d-1-\frac{d \gamma}{2s}} \approx -\A(t)^{2s/d} \A'(t) |v|^{-d-1-\frac{d\gamma}{2s}}
\end{align*}
which contradicts $\partial_t g(t_0,v_0) \le Q(f,f)(t_0,x_0,v_0)$ at
the contact point by picking $\Ao$ large enough.

\section{Relaxing partially the qualitative rapid decay assumption}
\label{s:no-decay}

This section discusses various ways of weakening the qualitative
assumptions made on the initial data
in~\ref{d:solutions-rapid-decay}. Observe first that if a clean local
existence and stability theory was available in
$H^{k}(\langle v \rangle)$ for hard potentials, and some $k,q$, then
it would be possible to use the generation of $L^1$ moments
conditionally to hydrodynamic bounds in the style of
\cite[Subsection~5.3.1]{GMM}, together with interpolation, to deduce
the qualitative pointwise moments. In the case of soft potentials such
a local existence and stability theory is available
in~\cite{MR3376931} but the propagation of $L^1$ moments conditionally
to hydrodynamic bounds is not available: if it was, an approximation
argument on the initial data (truncating its support) could be
performed. We postpone this discussion to another work.

Meanwhile we discuss here how to weaken the qualitative decay assumed
in Definition~\ref{d:solutions-rapid-decay} by approximation argument
on the barrier $g(t,v)$ used in the maximum principle arguments.

\subsection{Solutions without rapid decay and statement}

\begin{defn}[Classical solutions to the Boltzmann equation with mild
  decay] \label{d:solutions-no-decay} Given $T \in (0,+\infty]$, we
  say that a function
  $f : [0,T] \times \mathbb T^d \times \R^d \to [0,+\infty)$ is a
  \emph{classical solution to the Boltzmann equation
    \eqref{eq:boltzmann} with mild decay} if
\begin{itemize}
\item the function $f$ is differentiable in $t$ and $x$ and twice
  differentiable in $v$ everywhere;
\item the equation \eqref{eq:boltzmann} holds classically at every
  point;
\item The limit $\lim_{|v| \to \infty} f(t,x,v) = 0$ holds uniformly
  in $t \in [0,T]$ and $x \in \mathbb{T}^d$.
\end{itemize}
\end{defn}
%-----------------------------------------------------------------------
\begin{thm}[Pointwise moment bounds revisited]\label{thm:upper2}
  Let $\gamma \in \R$ and $s \in (0,1)$ satisfy $\gamma +2s \in [0,2]$
  and $f$ be a solution of the Boltzmann equation~\eqref{eq:boltzmann}
  as in Definition \ref{d:solutions-no-decay} such that
  $f(0,x,v)=f_0(x,v)$ in $\mathbb{T}^d \times \R^d$ and \eqref{eq:non-deg}
  holds. Then
\begin{enumerate}
\item \label{ii:1} If $\gamma \in (-2,0]$ and
  $q = d+1+\frac{\gamma d}{2s}$ if $\gamma <0$ or $q \in [0,q+1)$,
  then there exists $\A>0$ depending on $m_0,M_0,E_0, H_0, d$ and $s$
  such that 
  \[
    \forall \, t \in (0,T], \ x \in \mathbb{T}^d, \ v \in \R^d, \quad
    f(t,x,v) \le \A \left(1+t^{-\frac{d}{2s}} \right) \min
    \left(1,|v|^{-q} \right).
  \]
\item \label{ii:2} If $\gamma >0$ there exists a constant $\A>0$
  depending on $m_0,M_0,E_0,H_0,d$ and $s$, and a power $\beta>0$ such
  that
  \[ \forall \, t \in (0,T], \ x \in \mathbb{T}^d, \ v \in \R^d, \quad
    f(t,x,v) \le \A \left(1+t^{-\beta} \right) \min \left(1,|v|^{-d-1}
    \right).\]
\item \label{ii:3} If $\gamma \leq 0$ and $\gamma+2s < 1$, there
  exists $q_0$ depending on $d,s,\gamma, m_0,M_0,E_0,H_0$ such that for
  all $q \ge q_0$ and $f_0 \le C \min (1 , |v|^{-q})$ then there
  exists $\A$ depending on $C,m_0,M_0,E_0,H_0,q, d$ and $s$ such that
  \[
    \forall \, t \in [0,T], \ x \in \mathbb{T}^d, \ v \in \R^d, \quad
    f (t,x,v) \le \A \min \left(1, |v|^{-q} \right).
  \]
  \item \label{ii:4} If $\gamma > 0$, there exists $q_0 >0$  such that if
    \[ \lim_{|v| \to \infty} |v|^{-q_0} f(t,x,v) = 0,\] holds
    uniformly in $t \in [0,T]$ and $x \in \mathbb{T}^d$, then for all
    $q>0$, there exists constants $\A$ and $\beta>0$ depending on
    $m_0,M_0,E_0,H_0,q,d$ and $s$ such that
    \[
      \forall \, t \in (0,T], \ x \in \mathbb{T}^d, \ v \in \R^d,
      \quad f(t,x,v) \le \A \left(1+t^{-\beta} \right) \min
      \left(1,|v|^{-q}\right).
    \]
\end{enumerate}
\end{thm}
\begin{remark}
  \begin{enumerate}
  \item Note that for $\gamma >0$, we know from part (\ref{ii:2})
    that for all  $t>0$,
    \( \lim_{|v| \to \infty} |v|^{-q} f(t,x,v) = 0, \)
    for any $q <
    d+1$. The assumption in part (\ref{ii:4}) would be automatically
    true if $q_0 < d+1$. Unfortunately, it is hard to compute $q_0$
    explicitly from our proof.
\item \label{rem:gammaplus2s}
  The only purpose of the technical assumption $\gamma+2s<1$ in
  (\ref{ii:3}) is to handle the error term -- see
  $\eps |v|^{-d-1+\eps}$ in \eqref{e:g2} below. It is most likely not
  necessary. It is certainly not necessary for the a priori estimate
  if we knew that our solution decays faster than $|v|^{-d-2}$ at
  infinity.
\end{enumerate}
\end{remark}

The proof follows the same pattern as before. The only new difficulty
is to prove the existence of the first contact point, and avoid the
situation where it would appear asymptotically as $|v| \to \infty$. To
this purpose we modify the barrier functions used in
Section~\ref{s:proof} by adding arbitrarily small correctors:
$\tilde g = g + e$. The correctors are related to the decay known on
$g$, in order to ensure the existence of the first contact point.

\begin{itemize}
\item For parts (\ref{ii:1}), (\ref{ii:2}) and (\ref{ii:3}) with
  $q \leq d+1$, we use a constant corrector
  \begin{equation} 
    \label{e:g1} 
    \tilde g(t,v) = \A(t) \left(1 \wedge |v|^{-q}\right) +  e \quad
    \mbox{ with } e = \eps >0.
  \end{equation}

\item For part (\ref{ii:3}) with $q>d+1$ and $\gamma \leq 0$, we use 
  \begin{equation} \label{e:g2} \tilde g(t,v) = \A(t) \left(1 \wedge
    |v|^{-q} \right) + e(t,v) \quad \mbox{ with } e(t,v):=\eps(t) \left(1 \wedge
    |v|^{-d-1+\eta}\right)
\end{equation}
for certain choices of $\A(t)$ and $\eps(t)$ and $\eta>0$.

\item For part (\ref{ii:4}) with $\gamma > 0$, it is enough to
  consider $q > q_0$ and we use
  \begin{equation} \label{e:g3} g(t,v) = \A(t) \left(1 \wedge |v|^{-q}\right) +
    e(t,v) \quad \mbox{ with } \quad e(t,v) := \eps(t) \left(1 \wedge
    |v|^{-q_0}\right)
  \end{equation}
  for certain choices of $\A(t)$ and $\eps(t)$.
\end{itemize}

\subsection{Technical estimates on the collision operator}

The following results are variations of the corresponding results in
Section~\ref{s:estimates} when taking into account the correctors to
the barrier function. We define the decomposition
$Q_s = \mathcal{G} + \mathcal{B}_1 + \mathcal{B}_2 + \mathcal{B}_3$ as
before in~\eqref{eq:decomposition}-\eqref{eq:decomposition-bis}.
%----------------------------------------
\begin{prop}[Estimate of {$\mathcal{G}(f,g)$} useful for large
  $q$] \label{prop:G-bis2} Let $f$ be a non-negative function satisfying \eqref{eq:non-deg} and $g$
  given by \eqref{e:g1}, \eqref{e:g2} or \eqref{e:g3} with $q \ge
  0$. Then there exists a radius $R_q = C_R (1+q)$ so that
  \[
    \forall \, |v| \geq R_q, \quad
    \mathcal{G} (f,g)(v)
    \lesssim \begin{cases}
      - q^s \A  |v|^{\gamma-q} & \text{if $g$ is as in \eqref{e:g1}}, \\[2mm]
      - q^s \A  |v|^{\gamma-q} - \eps(t) |v|^{\gamma-(d+1)+\eta}  & \text{if $g$ is as in \eqref{e:g2}}, \\[2mm]
      - q^s \A  |v|^{\gamma-q} - q_0^s \eps(t) |v|^{\gamma-q_0}  & \text{if $g$ is as in \eqref{e:g3}}.
    \end{cases}
  \]
  where the constants $C_R >0$ and in the latter inequality are
  independent of $q$.
\end{prop}
%-----
\begin{proof} It is a straight forward modification of
  Proposition~\ref{prop:G-bis} adding an extra correction term. In the
  case $g$ is as in \eqref{e:g1}, note that the extra terms $+\eps$
  will cancel out in the upper bound for $(f(v') - f(v))$.
\end{proof}
%------------------------------------------------------------------------------
\begin{prop}[Estimate of {$\mathcal{G}(f,f)$} useful for not-so-large
  $q$] \label{prop:G2} Assume $f$ satisfies \eqref{eq:contact} for of
  the form \eqref{e:g1} or \eqref{e:g2} and $q \ge 0$. Then there
  exists $R_q=C_R(1+q)$ so that  
  \[ \forall \, v \in \R^d \ | \ |v| \ge R_q, \quad \mathcal{G}
    (f,g)(v) \lesssim_q - g(v)^{1+\frac{2s}d}
    |v|^{\gamma+2s+\frac{2s}d}.
  \] 
\end{prop}
%---------------------------------------------------------------------------------
\begin{proof}
  We follow the same ideas as in the proof of
  Proposition~\ref{prop:G}. We first analyse the range of values of
  $v$ where the inequality follows from Proposition~\ref{prop:G-bis2}.

  If $g$ is given by \eqref{e:g1}, then the estimate of
  Proposition~\ref{prop:G2} derives from Proposition~\ref{prop:G-bis2}
  for $|v| \ge 1$ and $\eps \in (0,1)$ such that
  \[ \frac{\eps}3 \le \A |v|^{-q} \lesssim_q |v|^{-(d+1)}. \] Indeed
  it implies
  $\eps^{1+\frac{2s}d} \lesssim  \A |v|^{-q - 2s - \frac{2s}d}$ and
  $(\A |v|^{-q})^{1+\frac{2s}d} \lesssim  \A |v|^{-q - 2s -
    \frac{2s}d}$ which in turn yields
  \[ g(v)^{1+\frac{2s}d} \lesssim \A |v|^{-q-2s - \frac{2s}d}\] and
  then the conclusion follows from Proposition~\ref{prop:G-bis2}.

  If now $g$ is given by \eqref{e:g2} or \eqref{e:g3}, then the
  estimate of Proposition~\ref{prop:G2} derives from
  Proposition~\ref{prop:G-bis2} as soon as $|v| \ge 1$,
  $\eps \in (0,1)$ and $g(v) \lesssim_q |v|^{-(d+1)}$. Indeed, we then
  have $g(v)^{\frac{2s}d} |v|^{2s + \frac{2s}d} \lesssim_q 1$ and the
  conclusion follows.

  We are left with two cases: (1) when $g(v) \gtrsim_q |v|^{-(d+1)}$ with
  $g$ given by \eqref{e:g1}, \eqref{e:g2} or \eqref{e:g3}, or
  (2) when $g(v) < \eps/3$ and $g$ is of the form \eqref{e:g1}.

  In both cases, we argue as in the proof of
  Proposition~\ref{prop:G-bis}. We pick $r>0$ such that
\[ \left|\Xi(t,x,v) \cap B_r\right| = \frac{4^2E_0}{|v|^2 g(v)}.\] 
and deduce 
\[
\mathcal G(f,g)(v)  \leq \int_{\Xi(t,x,v) \cap B_r \cap \left\{f(v+w)
    \le \frac{g(v)}{2}\right\}} \left[ \frac{g(v)}{2} - g(v+w)\right]
K_{\bar f}(v,v+w) \dd w.
\]
As in Proposition~\ref{prop:G-bis}, this is a useful estimate if
$g(v+w) > g(v)/2$ for $w \in \Xi \cap B_r$. If
$g(v) \gtrsim_q |v|^{-d-1}$, we end the proof as in
Proposition~\ref{prop:G-bis} (by choosing an appropriately large
constant $C_q$). If $g$ is given by \eqref{e:g1} and
$\A |v|^{-q} \le \frac\eps3$, then we have for all $w \in \R^d$ that
$g(v+w) \ge 3g(v)/4$. This also allows to continue the proof and
conclude.
\end{proof}

\begin{prop}[Estimate of {$\mathcal{B}_1(f,g)$} for all $q \ge
  0$] \label{prop:B1-2} Let $f$ be a non-negative function satisfying \eqref{eq:non-deg}. Let $g$ be of
  the form \eqref{e:g1}, \eqref{e:g2} or \eqref{e:g3} with $q \ge
  0$. Then for $|v| \ge 2$,
  \[ \mathcal{B}_1 (f,g)(v) \lesssim (1+q)^2 2^q |v|^{\gamma -2} g(v)\]
  with constant uniform in $q$.
\end{prop}

\begin{proof} It is a straight forward adaptation of the proof of Proposition~\ref{prop:B1}.
\end{proof}

\begin{prop}[Estimate of {$\mathcal{B}_2(f,f)$} for large
  $q$] \label{prop:B2-2} Let  $f$ be a non-negative function satisfying \eqref{eq:non-deg}. Assume $f \leq g$ for all $v \in \R^d$ and
  either $g$ is of the form \eqref{e:g2} with $\gamma+2s < 1-\eta$, or
  $g$ is of the form \eqref{e:g3} with $q_0 > d+\gamma+2s$. Assume
  further that $q> \gamma +2s +d$. Then for $|v| \ge 2$
  \[
    \mathcal{B}_2 (f,f) \lesssim \begin{cases}\displaystyle
      \frac 1 {q-(d+\gamma + 2s)}  |v|^{-q+\gamma} + \eps
      |v|^{-d-1+\eta +\gamma}
      & \text{ if $g$ is as in \eqref{e:g2}}, \\[3mm] \displaystyle
      \frac 1 {q-(d+\gamma + 2s)} |v|^{-q+\gamma} + \eps \frac 1
      {q_0-(d+\gamma + 2s)} |v|^{-q_0 +\gamma}
      & \text{ if $g$ is as
        in \eqref{e:g3}}.
    \end{cases}
  \]
\end{prop}

\begin{proof}
  It is the result of the same computation as in the proof of
  Proposition~\ref{prop:B2} but with the extra correction terms. The
  purpose of the assumptions $\gamma+2s < 1-\eta$ or
  $q_0 > d+\gamma+2s$ is to make sure the tail of the integral
  \[ \int_{v'_* \in v + (v'-v)^\bot} g(v'_*) |v-v_*|^{\gamma+2s +1}
    \dd v'_*\] is convergent (which was also the purpose of the
  assumption $q > \gamma+2s-d$).
\end{proof}

\begin{prop}[Estimate of {$\mathcal{B}_3(f,f)$} for large
  $q$] \label{prop:B3-2} Let  $f$ be a non-negative function satisfying \eqref{eq:non-deg}. Assume $f \leq g$ for all $v \in \R^d$ and
  $g$ of the form \eqref{e:g2} or \eqref{e:g3} and $q> d+\gamma
  +2s$. Then for all $|v| \ge 2$
  \[ \mathcal{B}_3 (f,f)(v) \lesssim_q \frac{1}{q-(d+\gamma+2s)} |v|^{\gamma -2} g(v).\]
  \end{prop}
  \begin{remark}
    The dependency in $q$ of the constant is explicit and can be
    tracked from the proof below.
  \end{remark}
  
\begin{proof}
  The proof is similar to that of Proposition~\ref{prop:B3} but takes
  the extra corrector term into account. Define for $|v'| < |v|/2$: 
\begin{equation*}
  I_3(v,v') := \int_{v'_* \in \; v + (v'-v)^\bot} \tilde \chi(v'_*) g(v'_*)
  |v-v'_*|^{\gamma+2s+1} \tilde b(\cos \theta) \dd v'_*
\end{equation*}
and decompose $v'_* = |v| (\hat v + \tilde u)$ and calculate as before
(the restriction $\tilde \chi$ imposes $|\hat v + \tilde u| > c_1(q)$)
\begin{align*}
& I_3(v,v') 
\lesssim   \A |v|^{-q+ \gamma +2s+d} \int_{\tilde u \in \; (v'-v)^\bot}
            \tilde \chi(v'_*) \left|\hat v + \tilde u \right|^{-q}
            |\tilde u|^{\gamma+2s +1}  \dd \tilde u \\
& + \eps \begin{cases} \displaystyle
|v|^{-(d+1)+\eta + \gamma +2s+d} \int_{\stackrel{\tilde u \in
    \; (v'-v)^\bot}{\left|\hat v + \tilde u \right| > c_q}}
\left|\hat v + \tilde u \right|^{-(d+1)+\eta} |\tilde u|^{\gamma+2s
  +1}  \dd \tilde u & \text{if $g$ is as in \eqref{e:g2},} \\[3mm] \displaystyle
|v|^{-q_0 + \gamma +2s+d} \int_{\stackrel{\tilde u \in \;
    (v'-v)^\bot}{\left|\hat v + \tilde u \right| > c_q}}
\left|\hat v + \tilde u \right|^{-q_0} |\tilde u|^{\gamma+2s +1}  \dd
\tilde u & \text{if $g$ is as in \eqref{e:g3}.}
\end{cases}
\end{align*}
This implies the following estimates which concludes the proof:
\begin{align*}
\mathcal{B}_3 (f,f)(v) \lesssim \begin{cases}\displaystyle
C_q \A  |v|^{-q+\gamma-2} + \eps |v|^{-(d+1)+\eta + \gamma -2}
& \text{when $g$ is as in \eqref{e:g2},}\\[3mm] \displaystyle
C_q \A  |v|^{-q+\gamma-2} + C_{q_0} \eps |v|^{-q_0 + \gamma -2} & \text{when $g$ is as in \eqref{e:g3}.}
\end{cases}
\end{align*}
\end{proof}

\begin{prop}[Estimate of {$\mathcal{B}_2(f,g) + \mathcal{B}_3(f,g)$} for
  not-so-large $q$]\label{prop:B23-2}
Let  $f$ be a non-negative function satisfying \eqref{eq:non-deg} and
  $g$ be of the form~\eqref{e:g1} and $q \in [0,d+1]$. Then for $|v| \ge 2$,
  \[ (\mathcal{B}_2 + \mathcal{B}_3)(f,g)(v) \lesssim \begin{cases}
      \A |v|^{-d-1+\gamma} & \text{if } q > d-1,\\
      \A |v|^{-d-1+\gamma} \ln (1+|v|) & \text{if } q = d-1,\\
      \A |v|^{-q-2+\gamma} & \text{if } q < d-1.
    \end{cases}
  \]
\end{prop}

\begin{proof}
  The proof is identical to that of Proposition~\ref{prop:B23}. Note
  that the extra constant corrector term $\eps$ cancels out in the estimate
  $f(v') - f(v) \leq g(v') - g(v)$.
\end{proof}

\begin{prop}[Estimate of $Q_{ns}(f,f)$] \label{prop:Q2-2} Assume $f$
  satisfies \eqref{eq:contact} with $g$ of the form \eqref{e:g1} or
  \eqref{e:g2}. Then for $\gamma \ge 0$ 
\[ Q_{ns}(f,f)(v) \lesssim  (1+|v|)^{\gamma} g(v), \]
while for $\gamma <0$,
\[
  Q_{ns}(f,f) \lesssim C_q g(v)^{1-\frac{\gamma}{d}} +
  (1+|v|)^\gamma g(v)
\]
for some constant $C_q$ depending on $q$. 
\end{prop}

\begin{proof}
  In the case $\gamma \geq 0$, the estimate
  $Q_{ns}(f,f) \lesssim |v|^{\gamma} f(v)$ implies the result follows
  for any form of the function $g$. In the case $\gamma < 0$, the
  proof of Proposition~\ref{prop:Q2} applies as soon as
  $g(v') \leq C_q g(v)$ whenever $|v'-v| < |v|/2$. This property is
  satisfied for all the variants of the function $g$ given by
  \eqref{e:g1}, \eqref{e:g2} or \eqref{e:g3}.
\end{proof}
%----------------------------------------------
\subsection{Proof of Theorem \ref{thm:upper2}}

\subsubsection{Proof of \textbf{part \eqref{ii:1}}.} It is identical
to the proof of part (\ref{i:1}) in Theorem \ref{thm:upper} but using
\[\tilde g(t,v) = \A(t) \min(1,|v|^{-q}) + \eps\] for $\eps > 0$ arbitrarily
small. We apply Propositions~\ref{prop:G-bis2}, \ref{prop:B1-2},
\ref{prop:B23-2} and \ref{prop:Q2-2} instead of
Propositions~\ref{prop:G-bis}, \ref{prop:B1}, \ref{prop:B23} and
\ref{prop:Q2} and we arrive to the same set of inequalities that imply
the contradiction.

\subsubsection{Proof of \textbf{part \eqref{ii:2}}.} We use the same
estimates as for part (\ref{ii:1}), which are not the same as the ones
used for part (\ref{i:2}) in Theorem \ref{thm:upper}. Set
$\tilde g(t,v) = \A(t) \min(1,|v|^{-q}) + \eps$ and
$\A(t) = \Ao t^{-\frac d {2s}}$, where $\Ao$ is a large constant
depending on $m_0, M_0, E_0, H_0, \gamma, s$ and $d$, to be determined
below, and $\eps$ is arbitrarily small. Apply
Propositions~\ref{prop:G-bis2}, \ref{prop:B1-2}, \ref{prop:B23-2} and
\ref{prop:Q2-2} at the point of contact $(t_0,x_0,v_0)$, for $|v_0|$
large enough:
\begin{align*}
  \mathcal G(f,f)(t_0,x_0,v_0) &\lesssim - |v_0|^{(\gamma+2s)+\frac{2s}d}
                         g(v_0)^{1+\frac{2s}d}
  && \text{from Proposition \ref{prop:G2}}, \\
  \mathcal B_1(f,f)(t_0,x_0,v_0) &\lesssim |v_0|^{\gamma-2} g(v_0)
  && \text{from Proposition \ref{prop:B1-2}}, \\
  (\mathcal B_2 + \mathcal B_3)(f,f)(t_0,x_0,v_0) &\lesssim |v_0|^{\gamma}
                                            g(v_0)
  && \text{from Proposition \ref{prop:B23-2}},\\
  Q_{ns}(f,f)(t_0,x_0,v_0) &\lesssim |v_0|^\gamma g(v_0)
  && \text{from Proposition \ref{prop:Q2-2}}.
\end{align*}
As before $|v_0|$ large can be imposed by taking $\A$ large, and the
first negative term dominates all other at large $|v_0|$ which
contradicts $\partial_t \tilde g(t_0,v_0) \le Q(f,f)(t_0,x_0,v_0)$ and
concludes the proof.

\subsubsection{Proof of \textbf{part \eqref{ii:3} in the case
    $\gamma \le 0$ and $q = d+1$}.}

Consider a function $\tilde g$ of the form \eqref{e:g1} with $q=d+1$
and $\eps>0$ arbitrarily small and $\A=\Ao$ large enough so that
$\tilde g(0,v) \geq f(0,x,v)$ everywhere (using the $L^\infty$ bound on
$f$). The first contact $(t_0,x_0,v_0)$, such that \eqref{eq:contact}
holds true, exists because $f$ goes to zero as $|v| \to
+\infty$. Using the $L^\infty$ bound and picking $\A$ large enough, we
can force $|v_0|$ to be arbitrarily large, and we can apply
Propositions~\ref{prop:G2}, \ref{prop:B1-2}, \ref{prop:B23-2} and
\ref{prop:Q2-2}:
\begin{align*}
  \mathcal G(f,f)(t_0,x_0,v_0) &\lesssim - |v_0|^{\gamma+2s+\frac{2s}d}
                         g(v_0)^{1+\frac{2s}d}
  && \text{from Proposition \ref{prop:G2}}, \\
  \mathcal B_1(f,f)(t_0,x_0,v_0) &\lesssim |v_0|^{\gamma-2} g(v_0)
  && \text{from Proposition \ref{prop:B1-2}}, \\
  (\mathcal B_2 + \mathcal B_3)(f,f)(t_0,x_0,v_0) &\lesssim |v_0|^{\gamma}
                                            g(v_0)
  && \text{from Proposition \ref{prop:B23-2}},\\
  Q_{ns}(f,f)(t_0,x_0,v_0) &\lesssim |v_0|^\gamma g(v_0) + C_q g(v_0)^{1-\frac{\gamma}{d}}
  && \text{from Proposition \ref{prop:Q2-2}}.
\end{align*}

Since
$|v_0|^{(\gamma+2s)+\frac{2s}d} \gtrsim (|v_0|^{-q} +
\eps)^{-\frac{\gamma+2s}{d+1} - \frac{2s}{d(d+1)}}$ uniformly as
$\eps \to 0$ ($|v_0|$ is not close to zero), the negative term
dominates the term $g(v_0)^{1-\frac{\gamma}{d}}$ by taking $\Ao$ large
enough, and we deduce for some constants $K,C>0$
\begin{align*} 
Q(f,f)(t_0,x_0,v_0) &\leq -K |v_0|^{\gamma+2s+\frac{2s}d} g(v_0)^{1+\frac{2s}d} + C |v_0|^{\gamma} g(t,v), \\
&\leq |v|^{\gamma} g(t,v) \left( -K  \Ao ^{\frac{2s}d} + C \right).
\end{align*}
We choose $\Ao$ large enough to achieve the contradiction
$Q(f,f)(t_0,x_0,v_0) <0$. 

\subsubsection{Proof of \textbf{part \eqref{ii:3} in the case
    $\gamma \le 0$ and $q$ large}.}
Having proved already that~\eqref{ii:3} holds when $q=d+1$, we now use
the corrected barrier $\tilde g$ as in~\eqref{e:g2}. The previous
subsubsection implies then $f(t,x,v) < g(t,v)$ when $v$ is
sufficiently large and therefore there is a first contact point
$(t_0,x_0,v_0)$. Take $\eps(t) = \eps_0e^{C_\eps t}$ in~\eqref{e:g2},
for $\eps_0>0$ arbitrarily small. As before we impose $|v_0|$ large
thanks to the $L^\infty$ by choosing $\Ao$ large enough, and we now
apply Propositions~\ref{prop:G-bis2}, \ref{prop:G2}, \ref{prop:B1-2},
\ref{prop:B2-2}, \ref{prop:B3-2} and \ref{prop:Q2-2}. Following the
same computations as in the proof of part (\ref{i:3}) of
Theorem~\ref{thm:upper}, the principal terms cancel out and we are
left with the terms derived from the correction term
$\eps(t) |v_0|^{-d-1+\eta}$. We get
\[
  Q(f,f)(t_0,x_0,v_0) \leq C \eps(t) |v_0|^{\gamma+d+1-\eta}.
\]
Since $\gamma\leq 0$ and $|v_0|$ is large, we have
$Q(f,f)(t_0,x_0,v_0) < C_\eps \eps(t)$ for some $C_\eps >0$. We
plug $C_\eps$ in the corrector
$\eps(t) = \eps_0 e^{C_\eps t}$ and achieve the contradiction.

\subsubsection{Proof of \textbf{part \eqref{ii:4}}.} We now use
$\tilde g$ as in \eqref{e:g3} with $\A(t) = \Ao t^{-\beta}$ with
$\beta:=\frac{q}{\gamma}-\frac{d}{2s}$ and
$\eps(t) = \eps_0 t^{-\beta_0}$ with
$\beta_0:=\frac{q_0}{\gamma}-\frac{d}{2s}$ and $\Ao$ large enough and
$\eps_0$ arbitrarily small and the exponents $q$ and $q_0$ large
enough, to be chosen later. The first contact point $(t_0,x_0,v_0)$
exists because of the convergence $|v|^{q_0} f(t,x,v) \to 0$ as
$|v| \to +\infty$ and the corrector term. We impose $|v_0|$ large
enough by taking $\Ao$ large enough, and we apply
Propositions~\ref{prop:G-bis2}, \ref{prop:B1-2}, \ref{prop:B2-2},
\ref{prop:B3-2} and \ref{prop:Q2-2}:
\begin{align*}
  \mathcal G(f,f)(t_0,x_0,v_0)
  &\lesssim - q^s \A(t) |v_0|^{-q+\gamma} - q_0^s \eps(t)
    |v_0|^{-q_0+\gamma}
  && \text{from Proposition } \ref{prop:G-bis2},\\
  \mathcal B_1 (f,f)(t_0,x_0,v_0)
  &\lesssim_{q,q_0} \A(t) |v_0|^{-q+\gamma-2} + \eps(t)
    |v_0|^{-q_0+\gamma-2}
  && \text{from Proposition } \ref{prop:B1-2}, \\
  \mathcal B_2 (f,f)(t_0,x_0,v_0)
  &\lesssim \frac 1 q \A(t) |v_0|^{-q+\gamma} + \frac 1 {q_0} \eps(t)
    |v_0|^{-q_0+\gamma}
  && \text{from Proposition } \ref{prop:B2-2},\\
  \mathcal B_3 (f,f)(t_0,x_0,v_0)
  &\lesssim_{q,q_0} \A(t) |v_0|^{-q+\gamma-2} + \eps(t)
    |v_0|^{-q_0+\gamma-2}
  && \text{from Proposition } \ref{prop:B3-2}, \\
  Q_{ns}(f,f)(t_0,x_0,v_0)
  &\lesssim \A(t) |v_0|^{-q+\gamma} + \eps(t) |v_0|^{-q_0+\gamma}
  && \text{from Proposition } \ref{prop:Q2-2}.
\end{align*}
The first negative term dominates all other term when $q$ and $q_0$
and $|v_0|$ are sufficiently large and we deduce 
\[ Q(f,f)(t_0,x_0,v_0) \lesssim -q^s \A(t) |v_0|^{-q+\gamma} - q_0^s
  \eps(t)|v_0|^{-q_0+\gamma}.\] We use that
$|v_0| \gtrsim_q \Ao^{1/q} t^{-1/\gamma}$ to get
\[
  Q(f,f)(t_0,x_0,v_0) \lesssim - q^s t^{-\beta-1} |v_0|^{-q} - q_0 ^s \eps_0
  t^{\beta-1} |v_0|^{-q_0}
\]
which yields a contradiction for $q \ge q_0$ large enough, and finishes the proof. 

\bibliographystyle{plain}
\bibliography{upper}
\end{document}